\documentclass{amsart}
\usepackage{amsfonts,amscd, amssymb}

\newtheorem{theorem}{Theorem}[section]
\newtheorem{lemma}[theorem]{Lemma}
\newtheorem{corollary}[theorem]{Corollary}
\newtheorem{proposition}[theorem]{Proposition}
\theoremstyle{remark}
\newtheorem{remark}[theorem]{Remark}
\theoremstyle{definition}

\newtheorem{example}[theorem]{Example}
\numberwithin{equation}{section}
\makeatother

\DeclareMathOperator{\Cdb}{{\mathbb C}}
\DeclareMathOperator{\Rdb}{{\mathbb R}}

\DeclareMathOperator{\Ndb}{{\mathbb N}}
\DeclareMathOperator{\eps}{\epsilon}

\begin{document}
\title[Noncommutative topology and Jordan operator algebras]{Noncommutative topology and Jordan operator algebras}

\author{David P. Blecher}
\address{Department of Mathematics, University of Houston, Houston, TX
77204-3008, USA}
\email[David P. Blecher]{dblecher@math.uh.edu}

 \author{Matthew Neal}
\address{Department of Mathematics and Computer Science,
Denison University, 100 West College St, Granville, OH 43023,
 USA}
\email[Matthew Neal]{nealm@denison.edu}

\keywords{Jordan operator algebra, Jordan Banach algebra, Open projection, Hereditary subalgebra, Approximate identity, JC*-algebra, Operator spaces, $C^*$-envelope, Real positivity, States, Noncommutative topology, Peak set, Peak interpolation}
\subjclass[msc2010]{Primary: 17C65, 46L07, 46L70,  46L85, 47L05, 47L30; Secondary  46H10, 46H70, 47L10, 47L75}

\date{\today} 
	
\thanks{We are grateful to Denison University and the Simons Foundation (Collaborative
grant 527078) for some financial support.}

\begin{abstract}  Jordan operator algebras are norm-closed spaces of operators
on a Hilbert space with $a^2 \in A$ for all $a \in A$.
 We study noncommutative topology, noncommutative peak sets and peak interpolation, 
and hereditary subalgebras of  Jordan operator algebras.   We show  that Jordan operator algebras 
present perhaps the most general setting for a `full' noncommutative topology in the $C^*$-algebraic sense of 
Akemann, L. G. Brown, Pedersen, etc, and as modified for not necessarily selfadjoint algebras by the authors with Read, Hay and
other coauthors.   Our breakthrough relies in part on establishing several strong 
variants of $C^*$-algebraic results
of Brown  relating to hereditary 
subalgebras, proximinality, deeper facts about $L+L^*$ for a left ideal $L$
in a $C^*$-algebra, noncommutative Urysohn lemmas, etc.     We also prove several other approximation results in $C^*$-algebras
and various subspaces of $C^*$-algebras, related to open and closed projections and 
 technical $C^*$-algebraic results of Brown.  
  \end{abstract}
\maketitle

\section{Introduction}   
Akemann's noncommutative topology  (see e.g.\ \cite{Ake2,APfaces}) for $C^*$-algebras
is a powerful tool that recasts the usual notions and results from topology, and generalizes them, in the setting
of orthogonal projections  in the bidual of the $C^*$-algebra.  Thus a projection $p$ in  the bidual $B^{**}$
 of a $C^*$-algebra $B$ is called {\em open} if it is the weak* limit of an increasing net of positive elements of $B$
(see \cite{P}).    Open projections are also precisely the support projections of hereditary subalgebras of $B$
(that is, the identity of the weak* closure of the subalgebra in $B^{**}$).   This  links noncommutative topology  with a crucial concept in modern $C^*$-algebra theory.
Then $q$ is {\em  closed} if $q^\perp = 1- q$ is open.      For a locally compact Hausdorff space
 $K$,
a projection  $q \in C_0(K)^{**}$  is open (resp.\ closed) if  and only if it is the canonical image (via the 
process of viewing elements of $C_0(K)^{*}$ as measures)  in $C_0(K)^{**}$ of the 
characteristic function of an  open (resp.\ closed) set in $K$.    Akemann, L. G. Brown, Pedersen, and others,
then go on to discuss compact projections, etc, generalize key results in topology such as the Urysohn lemma
to $C^*$-algebras, and
give many deep applications (see e.g.\ \cite{Brown}).

An (associative) {\em operator algebra} is a possibly nonselfadjoint closed subalgebra of $B(H)$, for a
complex Hilbert space $H$.    In a series of many papers (e.g.\ \cite{BHN, BRI, BRII, BRord, BNII, BBS, Bnpi}) 
 the authors,  with Read, Hay and
other coauthors,  generalized Akemann's noncommutative topology to such operator algebras.  Indeed following the 
lead of Hay's thesis \cite{Hay}, 
we fused this  noncommutative topology 
with the classical theory of peak sets, generalized peak sets, peak interpolation, etc, for function
algebras (see e.g.\ \cite{Gam}).  The latter topics are crucial tools for studying  classical algebras of functions. 

 By a {\em  Jordan operator algebra} we
 mean a  norm-closed  {\em  Jordan subalgebra} $A$ of a $C^*$-algebra,  
namely a norm-closed   subspace closed under the 
`Jordan product' $a \circ b = \frac{1}{2}(ab+ba)$. Or equivalently,
 with $a^2 \in A$ for all $a \in A$ (that is, $A$ is
closed under squares; the equivalence
 uses the identity $a \circ b = \frac{1}{2} ((a+b)^2 -a^2 -b^2)$).   This is an interesting class of operator spaces that seems to have had almost had no study until now (besides \cite{AS}) 
except in the selfadjoint case.
The selfadjoint case, that is,  closed selfadjoint subspaces of a $C^*$-algebra which are closed under squares, 
are exactly what is known in the literature as 
{\em JC*-algebras}, and these do have a large literature, see e.g.\ \cite{Rod, HS} for references).     

There are  many interesting examples 
of Jordan operator algebras (we collect some of these in Section \ref{Ex}).     
In a recent preprint \cite{BWj} the first author and Wang initiated  the theory of (possibly nonselfadjoint) Jordan operator   algebras.    Unfortunately progress in \cite{BWj}  could only proceed to a certain point 
because we were blocked by a couple of difficult issues.     In particular, the {\em noncommutative topology} 
and {\em  noncommutative peak set/peak interpolation/hereditary subalgebra} theory
was obstructed at an early stage.      The main contribution of the present paper
is the solution of these  difficult points, using variants of some deep $C^*$-algebra arguments
of Brown \cite{Brown}, thus  removing the blockage.  These new $C^*$-algebraic facts are 
employed in deep operator space variants of classical peak interpolation lemma, to give the Jordan algebra
variant of Hay's theorem (the main result of \cite{Hay}, see also \cite{Bnpi} for a 
simplification).    Thus we establish  the crucial link that
allows our algebras to plug directly into the $C^*$-algebraic noncommutative topology described in the first 
paragraph of the present paper, and also allows a theory of hereditary subalgebras to exist.   
 Using this breakthrough, we go on to show that, remarkably,
 everything relevant to these topics
for  associative operator algebras in a series of papers (e.g.\ 
\cite{BHN, BRI, BRII, BRord, BNII, BBS, Bnpi, Hay}), works in the Jordan case with very minor exceptions,  but
 much of it for 
the deep reasons just alluded to.
That is, Jordan operator algebras turn out to have a
good  noncommutative topology and  hereditary subalgebra theory.  Here `good' by definition (!) means what we with various coauthors have proved for associative subalgebras
of $C^*$-algebras.  This includes for example Urysohn lemmas, peak interpolation,  order theory (in the sense of e.g.\ \cite{BRord}, namely with regard to the `real positive' ordering), new relations with (e.g.\ the noncommutative topology of) enveloping $C^*$-algebras, lifting of projections, etc.     Indeed Jordan operator algebras seem to  be the most general possible setting for  this full 
noncommutative topology.

Even more than  was  seen in \cite{BWj},  it will come to light here that the theory of Jordan operator algebras is  astonishingly similar to that of associative operator algebras.    Since much of this
parallels the development found in several papers 
on the associative operator algebra case (see the list in the previous paragraph), there is 
quite a lot to do.   We  show how to    
generalize almost all of the remaining portion of the theory from the papers
just cited.
   We are able to include a rather large number of results 
in a relatively short manuscript since many proofs are similar to their
operator algebra counterparts, and thus we often need only discuss the 
new points that arise.    However we warn the reader that the proofs will become skimpier (and sometimes
nonexistent) towards the end, since it will be assumed that the reader by then has gained familiarity
with several of the basic tricks to `Jordanize' proofs.   Indeed many of these tricks and ideas 
occur already in \cite{BWj}.   

Several complementary facts and additional theory
will be forthcoming in a work in progress \cite{BWj2} and in  the second authors Ph.\ D.\ thesis \cite{ZWdraft}.  
In another direction, in \cite{BNjp} we study
contractive projections on Jordan operator algebras, finding 
Jordan variants of many of the results in \cite{BNp}, and some 
improvements on a couple of results of that paper, etc.   Indeed the latter project provided the impetus for \cite{BWj} and
 the present paper.

We now describe the structure of our paper.
  In Section 2 we give some
characterizations of Jordan operator algebras complementing the one given
in \cite[Section 2.1]{BWj}.  In Section 3 we list some examples of Jordan operator algebras.  

In 
Section 4 we prove the hard 
technical facts that comprise the main breakthrough of the present paper.
In part of Section 4 we  establish several variants of $C^*$-algebraic results
of L. G. Brown from \cite{Brown} relating to hereditary 
subalgebras, proximinality, deeper facts about $L+L^*$ for a left ideal $L$
in a $C^*$-algebra, etc.    We recall that a subspace $E$ of a Banach space $E$ is proximinal if there is always 
a (not necessarily unique) 
closest point in $E$ to every $x \in X$.  We also prove several other approximation results in $C^*$-algebras
and various subspaces of $C^*$-algebras, related to open and closed projections.  
The main application in Section 4 is the Jordan variant of Hay's main theorem
from \cite{Hay} (see also \cite{Bnpi} for a simplification of part 
of Hay's proof).   That is, if $A$ is a closed Jordan subalgebra of a $C^*$-algebra $B$ then
there is a bijective correspondence between hereditary subalgebras (HSA's for short) in $A$ and HSA's in $B$ with support projection in $A^{\perp \perp}$ (this is Theorem \ref{bhn29}).    By definition
a HSA in $A$ is a Jordan subalgebra $D$ of $A$ containing a 
Jordan contractive approximate identity such that $xAx \subset D$ for all $x \in D$.   We 
show that a subspace
$D$ of $A$ is a hereditary subalgebra,  if and only if  $D$ is of form 
$\{ a \in A : a = pap \}$ for a projection
$p \in A^{\perp \perp}$ which is open in Akemann's sense (see e.g.\ \cite{Ake2}) as a projection in $B^{**}$. 
Equivalently,   a projection
$p \in A^{\perp \perp}$ is what we called $A$-{\em open} in   \cite{BWj} (that is, there exists a net in $A$ with
$x_t = p x_t p \to p$ weak* in $A^{**}$) 
if and only if $p$ is 
open in Akemann's sense  in $B^{**}$.   If $A$ is an approximately unital  Jordan operator algebra then a projection $q$ in $A^{**}$ is {\em closed} if 
$1-q$ is open in $A^{**}$, where $1$ is the identity of $A^{**}$.   By the above, the latter is equivalent to $q$ being closed in Akemann's sense in $B^{**}$.   

  In Section 5 we give some related distance 
formulae in terms of limits along an approximate identity.   Again some of these 
seem new even in the $C^*$-algebra case.   

In Section 6 we give several initial
consequences of the results in  Section 4, to e.g.\ open and closed projections, hereditary subalgebras,
and quotients.  In Section 7 we give consequences for 
the theory of compact projections in  Jordan operator algebras.  We recall that a projection $q$ 
in the bidual of a $C^*$-algebra $B$ is {\em compact} in Akemann's sense if 
it is closed and there is an element $b \in {\rm Ball}(B)_+$ with $q = qb \, (= qbq)$.  This
is equivalent to $q$ being closed in $(B^1)^{**}$ (see e.g.\ \cite{APfaces} where compactness
is called `belonging locally to' $B$, also 
see \cite{BNII} and \cite[Section 6]{BRII} for the generalization
of compactness to nonselfadjoint associative operator algebras).  

In Section 8 we show that
the theory from \cite{BHN, BRI, BRII, BRord, BNII,Bnpi, Hay} of noncommutative peak sets, 
noncommutative peak interpolation, and some noncommutative Urysohn lemmas, etc, all generalize
to Jordan operator algebras.   Noncommutative peak interpolation and  Urysohn are 
important in such settings in some large part because these techniques allow one to build 
elements in $A$ which have certain prescribed properties or behavior  on Akemann's  noncommutative 
generalizations of closed sets, i.e.\ closed projections in the bidual of the $C^*$-algebra (see \cite{Bnpi} for
a survey of this topic  
for associative operator algebras). 
In $C^*$-algebras one has a good functional calculus and one can 
build elements using this calculus and other tricks we are familiar with in $C^*$-algebras, but in more general 
algebras the  functional calculus is not good enough and so often 
one has to have other devices for there to be a hope 
of doing such building.   

Finally in Section 9 we collect several miscellaneous results
in the  noncommutative topology of Jordan operator algebras, mostly Jordan versions of results from 
\cite[Section 6]{BRII}, \cite[Section 6]{BRord}, and \cite[Section 2]{blueda}. 
This includes for example a strict Urysohn lemma, and a Urysohn lemma where the interpolating elements
are `nearly positive' (that is,  as close as we like to 
a contraction that is positive in the usual $C^*$-algebraic sense).  

We remark that we focus on nonselfadjoint Jordan operator   algebras in the present paper and \cite{BWj}.
Some of our results and theory are new even in the selfadjoint  case, i.e.\ for JC*-algebras, however in that case
one would want to replace  real positive elements in statements and proofs
by elements that are positive in the usual sense, which would change those proofs a bit.   We have not taken the trouble to do this here, and leave it to the
interested reader.   See also \cite{FPP0,FPP} 
  for some results along these lines in a more general setting. 
 
 In the remaining part of Section 1 we give some background
and notation.   For background  on operator spaces and associative operator algebras we refer the reader
to \cite{BLM,Pisbk}, and for $C^*$-algebras  the reader could consult e.g.\ 
\cite{P}.   For Sections 6--9 the reader will also want to consult \cite{BWj} frequently
 for background, notation, etc, and will often be referred 
to that paper for various results that are used here.    

  For us a {\em projection}
is always an orthogonal projection.   The letters $H, K$ are reserved for Hilbert spaces.
If $X, Y$ are sets, then $\overline{XY}$ denotes the
closure of the span of products of the form $xy$ for $x \in X, y \in
Y$.      A  (possibly nonassociative) normed algebra $A$  is {\em unital} if it has an identity $1$ of norm $1$, a map $T$ 
is unital if $T(1) = 1$.    We say that $X$ is a {\em unital-subspace} (resp.\ {\em unital-subalgebra}) of
such $A$ if it is a subspace (resp.\ subalgebra) and $1_A \in X$.  We write $X_+$ for the positive operators (in the usual sense) that happen to
belong to $X$.   We write $M_n(X)$ for the space of $n \times n$ matrices 
over $X$, and of course $M_n = M_n(\Cdb)$.  

Jordan subalgebras of commutative (associative)
 operator algebras are of course ordinary (commutative associative)
operator algebras on a Hilbert space, and the Jordan product is the 
ordinary product.    In particular if $a$ is an element in a 
Jordan operator algebra $A$ inside a $C^*$-algebra  $B$,
 then the closed Jordan algebra
generated by $a$ in $A$ equals the closed Banach algebra
generated by $a$ in $B$.    We write this as oa$(a)$.    
A {\em Jordan homomorphism} $T : A \to B$ between
Jordan algebras
is of course  a linear map satisfying $T(ab+ba) = T(a) T(b) + T(b) T(a)$ for $a, b \in A$, or equivalently,
that $T(a^2) = T(a)^2$ for all $a \in A$ (the equivalence follows by applying $T$ to $(a+b)^2$). 
If  $A$ is a Jordan operator subalgebra  of $B(H)$, then the {\em diagonal}
$\Delta(A) = A \cap A^*$  is a JC*-algebra.   If $A$ is
unital then as a  a JC*-algebra $\Delta(A)$ 
  is independent of the Hilbert space $H$  (see the third paragraph of \cite[Section 1.3]{BWj}).   An element $q$ in a Jordan operator algebra $A$
 is called  a {\em projection} if $q^2 = q$ and $\| q \| = 1$
(so these are just the orthogonal projections on the 
Hilbert space $A$ acts on, which are in $A$).    Clearly $q \in \Delta(A)$.

A projection $q$ in a Jordan operator algebra $A$ will be called
{\em central} if $qxq = q \circ x$ for all $x \in A$.  This is equivalent to
$qx = xq = qxq$ in any $C^*$-algebra containing $A$ as a
Jordan subalgebra, by the first labelled equation in \cite{BWj}.     It is also equivalent to that $q$ is central in any generated
(associative) operator algebra, or in a generated $C^*$-algebra.   This notion is independent of the particular generated
(associative) operator algebra since it is captured by the intrinsic formula $qxq = q \circ x$ for $x \in A$.  
For a projection $q$ in $A$ we have $q \circ x = 0$ if  and only if $x = (1-q) x (1-q)$, 
as may be  seen by considering the $2 \times 2$ matrix picture 
of $x$ with respect to $q$.      We will often use the well known  fact or exercise that 
$qx = q$ iff $xq = q$ iff $qxq = q$ if $x$ is a contractive operator.    For this reason many of the expressions
of the form $qx = q$ in the papers that we are generalizing may be replaced by  $qx q= q$, which makes sense in the 
Jordan setting.   Indeed in a Jordan algebra $aba = 2 (a \circ b) \circ a - a^2 \circ b$ (see  the second labelled equation in \cite{BWj}),
where $\circ$ is the Jordan product. 

A 
{\em Jordan  contractive approximate identity}
(or {\em J-cai} for short) for $A$,  is a net
$(e_t)$  of contractions with $e_t \circ a \to a$ for all $a \in A$.
A {\em  partial cai} for $A$ is a net consisting of real positive  elements that acts as a cai  (that is, a contractive approximate identity)
for the ordinary product in
every $C^*$-algebra  which contains and is generated by $A$ as a closed Jordan subalgebra.    If a 
  partial cai for $A$ exists then $A$ is called {\em approximately unital}.  
It is shown in \cite[Section 2.4]{BWj}   that if $A$ has a J-cai then it has a partial cai.  

We recall that 
every Jordan operator algebra $A$ has a  unitization $A^1$ 
which is unique up
to isometric Jordan homomorphism (see
 \cite[Section 2.2]{BWj}).    If $A$ is approximately unital the unitization $A^1$ 
 is unique up completely isometric Jordan homomorphism by \cite[Proposition 2.12]{BWj}.   The latter
is not true in general  if $A$ is not   approximately unital (a two 
dimensional Hilbertian example is given in
\cite{BWj2}), but fortunately in the present paper the isometric
case suffices.

A {\em state} of an approximately unital Jordan 
operator algebra $A$ is a functional with $\Vert \varphi \Vert = \lim_t \, \varphi(e_t) = 1$
for some (or every) J-cai $(e_t)$ for $A$.  These extend to states of the unitization $A^1$.  They also
extend to a state (in the $C^*$-algebraic sense) on any $C^*$-algebra $B$ generated by $A$, and conversely
any state on $B$ restricts to a state of $A$.  See \cite[Section 2.7]{BWj}
for details.

If $A$ is a Jordan subalgebra of a $C^*$-algebra $B$ then $A^{**}$ with its Arens product 
is a Jordan subalgebra of the von Neumann algebra $B^{**}$ (see \cite[Section 1]{BWj}).   Since 
the diagonal $\Delta(A^{**})$ is 
a JW*-algebra  (that is a weak* closed JC*-algebra), it follows that $A^{**}$ is closed under meets and joins of projections.

Because of the uniqueness of unitization up to isometric isomorphism, for a Jordan operator algebra $A$ 
we can define unambiguously ${\mathfrak F}_A = \{ a \in A : \Vert 1 - a \Vert \leq 1 \}$.  Then 
 $\frac{1}{2} {\mathfrak F}_A = \{ a \in A : \Vert 1 - 2 a \Vert \leq 1 \} \subset {\rm Ball}(A)$.
Similarly, ${\mathfrak r}_A$, 
the {\em real positive} or {\em accretive} elements in $A$, may be defined as the 
the set of  $h \in A$ with  Re $\varphi(h) \geq 0$ for all states $\varphi$ of $A^1$.
This is equivalent to all the other usual conditions 
characterizing accretive elements as we said in \cite[Section 2.2]{BWj}.  
We have for example  ${\mathfrak r}_A = \{ a \in A : a + a^* \geq 0 \}$, where 
the adjoint and sum here is in (any) $C^*$-algebra  containing $A$ as a
Jordan subalgebra.   We also have
${\mathfrak r}_A = \overline{ \Rdb_+ {\mathfrak F}_A}$.   If $A$ is a Jordan subalgebra of a Jordan operator algebra $B$
then ${\mathfrak F}_A = {\mathfrak F}_{B} \cap A$
and ${\mathfrak r}_A = {\mathfrak r}_{B} \cap A$.    
Elements $x$  in ${\mathfrak r}_A$ have $n$th roots in the closed (associative) operator algebra oa$(x)$ generated 
by $x$, for all $n \in \Ndb$, and $x^{\frac{1}{n}} \to s(x)$ weak*, where 
$s(x)$ is the {\em support projection} of $x$.   See \cite{BWj} for more details and properties of $s(x)$, although
$s(x)$ lives in the associative  operator algebra oa$(x)^{**}$, so can be treated by the methods in the cited
papers for the associative  operator algebra case.      The same principle applies to the {\em peak projection}  $u(x)$
of an element $x \in {\rm Ball}(A)$: this is the weak* limit in the bidual of $(x^n)_{n \in \Ndb}$ in the case that 
this weak* limit  exists and is nonzero \cite[Lemma 1.3]{BRII}. 
 We say in this case that $x$ {\em peaks} on the projection $u(x)$.
 That is, we may usually directly apply the results about $u(x)$ 
from  the cited
papers for the associative operator algebra case. We define and discuss peak projections more adequately at the start of Section \ref{Peak}.    For now we simply recall from  that 
in a unital operator algebra $A$ we have the relation $u(x) = 1 - s(1-x)$ if $u(x)$
is the peak projection of  $x \in {\rm Ball}(A)$.  
 This  of course is true in the Jordan case too,
since as we said $u(x)$ and $s(1-x)$   may be computed in the   associative operator algebra generated by $x$ and $1$.
 In the case that $a \in \frac{1}{2} {\mathfrak F}_A$ then $u(a) = {\rm w^*lim}_n \, a^n$ is a projection in $A^{\ast\ast},$ and is the peak for $a$ in the sense that $q=u(a)$ satisfies all the equivalent conditions in 	\cite[Lemma 3.1]{BNII}. Also in this case, $u(a)\leq s(a),$ and $u(a^n)=u(a^{1/n})=u(a)$ if $n\in \Ndb.$ And $u(a)\neq 0$  if and only if  $\Vert a\Vert=1.$  See e.g.\ Corollary 3.3 in \cite{BNII}.

We say that a subspace $D$ of a Jordan  operator algebra is an {\em  inner ideal} (in the Jordan sense) if  $aAa \subset D$  for any $a \in D$.  Equivalently (by replacing $a$ by $a \pm c$), if $a,c \in D$
and $b \in A$ then $abc  + cba \in D$.      {\em Hereditary subalgebras}  (in the Jordan sense), or {\em HSA's} for short, are the Jordan subalgebras possessing a Jordan cai which are inner ideals in the Jordan product sense.

We write $[x,y,z] = xyz+zyx$  for $x, y, z$ in a Jordan algebra.
   For a projection $q$ in a 
unital operator algebra $M$
the {\it Peirce 1-projection} on $M$ is $P_1 (x) = [q,x,q^{\perp}]$ for $x \in M$.
This  is  a contraction, as may easily be seen from the $2 \times 2$ matrix picture of $x$ with respect 
to $q$.   The    Peirce 1-space of $q$ is $M_1(q) = \{ P_1(x) : x \in M \}$.   Thus 
$B^{**}_1(q) = \{ P_1(x) : x \in B^{**} \}$ for $C^*$-algebra $B$ and projection $q \in B^{**}$.
 The    {\it Peirce 0-space} of $q$ is  $M_0(q)  = q^\perp M q^\perp$, and the    {\it Peirce 2-space} of $q$ is  $M_2(q)  = q M q$; with obvious modifications if $M = B^{**}$ for $C^*$-algebra $B$.   If $q$ is a closed 
projection in $B^{**}$ then $B^{**}_0(q) \cap B = 
q^\perp B^{**} q^\perp \cap B = \{ b \in B : b = q^\perp B^{**} q^\perp \}$ is a hereditary subalgebra of $B$ which is
weak* dense in $B^{**}_0(q)$.    Indeed this characterizes closed projections.

\section{Characterizations of Jordan operator algebras}

In \cite[Section 2.1]{BWj} an abstract operator space
 characterization of Jordan operator algebras
with an identity or `approximate identity' is given.   
The following very different
characterization of Jordan operator algebras (which does
not assume the existence of any kind of identity or approximate identity)
is the Jordan algebra variant of the Kaneda-Paulsen theorem from 
\cite{KP}; but with a proof that is modeled on the quick proof of the 
latter theorem from \cite[Theorem 5.2]{BNmetric}.   In the following result,
$I(X)$ is the injective envelope \cite[Chapter 4]{BLM}, which is known to be a $C^*$-algebra if $X$ is
a unital operator space (see \cite[Corollary 4.2.8]{BLM}).   The $C^*$-algebra generated by $X$ inside
$I(X)$ is called the $C^*$-envelope 
$C^*_{\rm e}(X)$.
 
\begin{theorem} \label{KP}  
Let  $X$ be an operator space.  The
possibly nonassociative algebra products on $X$ for which
there exists a completely isometric Jordan homomorphism from $X$
onto a Jordan operator algebra, are in a correspondence
with the
 elements $z \in {\rm Ball}(I(X))$
such that $x z^* x
\subset X$ for all $x \in X$.  For such $z$ the associated Jordan
 operator algebra product on $X$ is $\frac{1}{2} 
(x z^* y + y z^* x)$ for $x, y \in X$
(viewing $X \subset I(X)$).
\end{theorem}   \begin{proof}  The one direction, and the last
statement, are similar to Remark 2 on p.\ 194 of \cite{BRS}, and work with $I(X)$ replaced by any 
TRO $Z$ containing $X$.   That is, suppose that $X \subset Z \subset B(H)$  for some Hilbert space $H$
with $Z Z^* Z \subset Z$ within
$B(H)$.  Suppose that $z \in {\rm Ball}(Z)$ and $x z^* x
\subset X$ for all $x \in X$.  View
$I(X)$ as a TRO in $B(H)$  for some Hilbert space $H$
(that is, we have $I(X) I(X)^* I(X) \subset I(X)$ within
$B(H)$, and $X \subset I(X) \subset B(H)$.   
See \cite{Ham} or 4.4.2 in
 \cite{BLM}).
Notice that the map
$$\Psi(x) =  \left[ \begin{array}{ccl}
x z^* & x(1- z^* z)^{\frac{1}{2}}  \\ 0 & 0 \end{array}
 \right] , \qquad x \in X, $$ 
is a complete isometry into $M_2(B(H)) \cong B(H^{(2)})$.
Indeed $\Psi$ effectively multiplies a copy of $x \in X$ by the coisometry
$[ \;  z^* \; \;   (1-z^* z)^{\frac{1}{2}} \; ],$
  and multiplication on the right 
by a coisometry is completely isometric.
Next, $\Psi(x)^2 = \Psi(x z^* x)$ for $x \in X$.  Hence  
the range of $\Psi$ is a Jordan subalgebra of $B(H^{(2)})$
and  $\Psi$ is a completely isometric Jordan homomorphism if we
equip $X$ with  product $\frac{1}{2} 
(x z^* y + y z^* x)$ for $x, y \in X$.    

For the other direction,
if $X$ is a Jordan subalgebra of $B(H)$,
then we can view
$X \subset I(X) \subset B(H)$ as above, and by injectivity there exists
a completely contractive projection $P$ from $B(H)$ onto
$I(X)$.  Set $z = P(1)$.   By a theorem
of Youngson \cite[Theorem 4.4.9]{BLM} 
 we have $$x^2 = P(x1^*x) = P(x P(1)^* x) = x z^* x, \qquad x,y \in X,$$
where the last expression  $x z^* x$ is a product with respect to the 
ternary multiplication in $I(X) = {\rm Ran}(P)$ (see e.g.\ 4.4.2 and 4.4.7
in \cite{BLM}).
  \end{proof}

   In the Kaneda-Paulsen theorem \cite{KP} the correspondence between
the product on $X$ and the element $z$ in the injective envelope is bijective.
This follows e.g.\ from \cite[Proposition
4.4.12]{BLM} and its `right-hand version': if $X z^* X = (0)$ then $X z^* = (0) =
X z^* z$, so $z^* z = 0 = z$.  In our Jordan case we do not see yet if 
 the correspondence between
the Jordan product on $X$ and the element $z$ in the injective envelope is  necessarily bijective.
In particular we do not see why $x z^* x = 0$ for all $x \in X$ need  imply that
$z = 0$ in $I(X)$, even if $X = I(X)$.  
 This difficulty disappears in the `unital case' in the 
next result, which is modeled on \cite[Corollary 5.3]{BNmetric}:

\begin{corollary} \label{kpc}
Let  $(X,u)$ be a unital operator space in the sense e.g.\ of
{\rm \cite{BNmetric}}.  The
possibly nonassociative algebra products on $X$ for which
there exists a completely isometric Jordan homomorphism from $X$
onto a Jordan operator algebra, are in a bijective correspondence
with the elements $w \in {\rm Ball}(X)$ such that
$x w x \subset X$ (multiplication taken in the 
$C^*$-algebra which is the $C^*$-envelope of $(X,u)$ (or 
which is the injective envelope $I(X)$ with its unique
$C^*$-algebra product for which $u$ is the identity).
For such $w$ the associated
Jordan operator algebra product on $X$ is $\frac{1}{2}(x w y + y wx)$
for $x, y \in X$.
\end{corollary} 

  \begin{proof}  In this setting
$I(X)$ may be taken to be a $C^*$-algebra, containing the $C^*$-envelope 
$C^*_{\rm e}(X)$ as a $C^*$-subalgebra,
with common identity $u$ (see Corollary 4.2.8 (1) and 4.3.3 in \cite{BLM}).  If $z$ is as in Theorem \ref{KP}
then $$w = z^* = u z^* u \in \{ x z^* x : x \in  X \} \subset X .$$
Thus the existence of the correspondence in our corollary follows from Theorem \ref{KP}.  
If $w' \in X$ was another element with
$x w' x = x w x$ for all $x \in X$, then setting $x = u$ shows
that $w' = w$.   Thus  the correspondence is bijective between
the Jordan product on $X$ and the element $w$ in the statement of the corollary. 
\end{proof}

Our  characterizations of Jordan operator algebras immediately have consequences like the following:

\begin{corollary} \label{apch}  Let $A$ be a Jordan  operator algebra, and
$P : A \to A$ a completely contractive projection.
  \begin{itemize} \item  [(1)] The range of $P$ with product
$P(x \circ y)$, is  completely   isometrically Jordan isomorphic to a Jordan operator algebra.  \item  [(2)]  If $A$ is an associative 
operator algebra then the range of $P$ with product $P(x  y)$, is  completely isometrically algebra isomorphic to an
associative operator algebra.
\item  [(3)]  If $A$ is unital and $P(1) = 1$ then the range of  $P$, with product
$P(x \circ y)$,
is unitally completely isometrically Jordan isomorphic to a unital  Jordan operator algebra.
  \end{itemize}
 \end{corollary} 

  \begin{proof}  (3) \  If $A$ is a unital Jordan subalgebra of $B(H)$, extend 
$P$ to a complete contraction $T : B(H) \to B(H)$.  Define $m(a,b) = T(ab)$, and note that 
$$\frac{1}{2}(m(a,b) + m(b,a)) = T(a \circ b) = P(a \circ b) \in P(A) , \qquad a, b \in P(A) .$$
 Also,  $m(1,a) = m(a,1) = a$ for $a \in P(A)$.    Thus we have verified the 
conditions of \cite[Theorem 2.1]{BWj}, so that $P(A)$ is unitally completely 
  isometrically Jordan isomorphic to a unital  Jordan operator algebra. 

(1) \ We sketch a proof of this using Theorem \ref{KP} twice.  By that result if $A$ is the  Jordan operator algebra there
exists $w \in {\rm Ball}(I(A))$ with $x^2 = x w^* x$, 
the latter being the ternary product on $I(A)$.
As in the proof of e.g.\ \cite[Theorem 4.2.9]{BLM} we can 
extend $P$ to a completely contractive surjection
$\tilde{P} : I(A) \to I(P(A))$, and extend the inclusion $P(A) \subset A$  to a complete
isometry $\tilde{i} : I(P(A)) \to I(A)$, such that $Q = \tilde{i} \circ \tilde{P}$ is a completely contractive projection
from $I(A)$ onto a copy $Z$ of $I(P(A))$.    We may regard $\tilde{i}$ as a completely isometric
isomorphism $\kappa$ from $I(P(A))$ onto $Z$.   By Youngson's theorem (e.g.\ \cite[Theorem 4.4.9]{BLM}) we have that 
$Z$ is a TRO with ternary product $Q(x y^* z)$, and 
$Q(x Q(w)^* x) = Q(x w^* x)  = Q(x^2) = P(x^2)$ for $x \in P(A)$, where $x^2$ denotes the square in $A$.   
Apply   $\kappa^{-1}$, which is a ternary isomorphism by e.g.\
  \cite[Corollary 4.4.6]{BLM},  with $\kappa^{-1}(Q(w)) =  \kappa^{-1}(\tilde{i}(\tilde{P}(w))) =  \tilde{P}(w)$.  We see that $P(x^2) = x \tilde{P}(w)^* x$, where the latter is the ternary product in $I(P(A))$ of $x$, $\tilde{P}(w)$, and $x$.
Now appeal to Theorem \ref{KP} with $A$ replaced by $P(A)$, to see that the latter 
is  completely   isometrically Jordan isomorphic to a   Jordan operator algebra.   

(2) \ Similar to the proof of (1) but replacing Theorem \ref{KP} with
the Kaneda-Paulsen theorem, and replacing the final $x$ in any products ending with $x$
 by $y$.  Thus for example $x^2$ for $x \in P(A)$ becomes: $xy$ for $x,
y \in P(A)$.    
 \end{proof}

\begin{corollary} \label{minj}  The unital Jordan  operator algebras that are minimal operator spaces (see e.g.\
{\rm 1.2.21 in \cite{BLM}}), are the unital 
function algebras.  A possibly nonunital Jordan  operator algebra that is a minimal operator space is an associative
commutative operator algebra.  Indeed it is (completely) isometrically isomorphic to a subspace $E$ of $C(K)$ for a compact space $K$,
with $E$ equipped with multiplication $(f,g) \mapsto fg h$ for some fixed
 $h \in {\rm Ball}(C(K))$.   (In this case $\overline{E h}$ is a function algebra.)  \end{corollary} 

  \begin{proof}   If $A$ is a unital Jordan operator algebra then $A$ is a unital Jordan subalgebra of its injective envelope $I(A)$ 
by a tiny modification of the proof of
  \cite[Corollary 4.2.8 (1)]{BLM}.  Now $I(A)$ 
is a commutative $C^*$-algebra  $C(K)$ if $A$ is minimal  (see e.g.\ 4.2.11 in \cite{BLM}).   Hence $A$ is an 
associative subalgebra of $C(K)$.   Similarly, if $A$ is a nonunital Jordan  operator algebra and  minimal operator space then $I(A) = C(K)$ linearly (completely) isometrically.
By Theorem \ref{KP} the Jordan product on $A$ is  $fg \bar{k}$ for some fixed
 $k \in {\rm Ball}(C(K))$, which is an associative
commutative algebra product.  Indeed  Remark 2 on p.\ 194 of \cite{BRS} shows that $A$  is an associative
 operator algebra with this product.  Clearly $A \bar{k}$ is a function algebra.  \end{proof}

\section{Examples of Jordan operator algebras}  \label{Ex}

In \cite{BWj} we mentioned that besides the associative operator algebras and JC*-algebras, one may obtain many examples of Jordan operator algebras by considering 
$\{ a \in A : \pi(a) = \theta(a) \}$ for a homomorphism $\pi : A \to A$ and an antihomomorphism $\theta : A \to A$, for an associative operator algebra $A$. 
Or  consider $\{ (\theta(a), \pi(a)) \in C  \oplus^\infty D \}$
for a homomorphism $\pi : A \to C$ and antihomomorphism $\theta : A \to D$
for associative operator algebras $C, D$.    
 In \cite{BNjp} we show that the range of various natural classes
of contractive projections on operator algebras
are Jordan operator algebras.   We gave a few more examples of Jordan operator algebras in \cite{BWj},
but proceed now to supplement the list with many others.

\begin{example}\label{gens}   A common way to uncover interesting operator spaces is to look 
at the linear span of the generators in well known $C^*$-algebras.  See e.g.\  \cite{Pisbk}.   One may vary 
this by looking at the smallest Jordan operator subalgebra of these $C^*$-algebras containing the 
generators.   For example one may look at the smallest closed Jordan subalgebra of the (full or
reduced) $C^*$-algebra
of the free group containing the generators (or the generators and their inverses),
this will be contained in the (relevant associative) `free semigroup operator algebra'.  
Other examples include:  the smallest closed Jordan subalgebra of the Cuntz algebra ${\mathcal O}_n$ containing the $n$ canonical generators, the smallest closed Jordan algebra containing a free semicircular family (see e.g.\ \cite[Theorem 9.9.7]{Pisbk}), or one could consider the smallest weak* 
closed Jordan subalgebra of the group von Neumann algebra
of the free group on $n$ generators, containing the generators (or the generators and their inverses).  
Some of these examples seem worth exploring in detail. 
\end{example}

\begin{example} \label{jajcst} Jordan operator algebras coming from JC$^*$-triples:
Let $Z$ be a JC$^*$-triple, that is a closed subspace of a $C^*$-algebra,
or of $B(K,L)$ for Hilbert spaces $K,L$, satisfying $x y^* z + z y^* x \in Z$ for all $x, y, z \in Z$.
JC$^*$-triples (also called $J^*$-algebras in \cite{Har2}) are well studied and there are many examples in the literature (sometimes 
in the  context of JB*-triples or operator spaces, see e.g.\ the works of Bunce and Timoney cited here and references therein, \cite{N2,HN}, etc).  
Fixing $y \in  {\rm Ball}(Z)$, we may define a bilinear map 
$m_y : Z \times Z \to Z$ by $m_y(x,z) = \frac{1}{2}(x y^* z + z y^* x)$.
By the argument in the proof of Theorem \ref{KP}, the map
$$\Psi(x) =  \left[ \begin{array}{ccl}
x y^* & x(1- y^* y)^{\frac{1}{2}}  \\ 0 & 0 \end{array}  \right] , \qquad x \in Z, $$
is a complete isometry from $Z$ into $M_2(B(H)) \cong B(H^{(2)})$,
if the $C^*$-algebra above is in $B(H)$ (in the $K,L$ situation replace $H^{(2)}$ here by $L \oplus K$).
Again $\Psi(x)^2 = \Psi(x y^* x)$ for $x \in Z$, so that
the range of $\Psi$ is a Jordan subalgebra of $B(H^{(2)})$
and  $\Psi$ is a completely isometric Jordan homomorphism if we
equip $Z$ with  Jordan product $m_y(x,z)$ for $x, y \in Z$.
Thus $Z$ with product $m_y$ is a Jordan operator algebra.   It will not usually have an identity, but if one
wants a unital example then one may unitize.

Some particularly interesting examples arise when considering Hilbertian JC$^*$-triples (by Hilbertian here we mean linearly 
isometric to a Hilbert space).    See e.g.\ \cite{N2,NR05,BFT,BT}, and references therein, where these authors also classify
Hilbertian JC$^*$-triples, and in particular contractively complemented Hilbertian JC$^*$-triples.   
We mention several very concrete examples.   First, consider the $n \times n$ antisymmetric matrices
(so that $a^T = -a$), a JC$^*$-triple.  
  When $n = 3$ this is a Hilbertian operator space of dimension 3, a Cartan factor of type 2, with three
  canonical basis elements.   By looking at the canonical basis elements and their relations
one can see that this space is (algebraically)  $J^*$-isomorphic to $B(\Cdb,\Cdb^3)$ in the sense of  \cite{Har2}, 
hence is isometric to that 
Hilbert space by  results in that paper (see also \cite{NR03} for more details about this space).  
Taking $y$ to be the canonical basis element $E_{13}-E_{31}$ the product
$m_y$ defined above gives an interesting example of a 3 dimensional Hilbertian Jordan operator algebra.  
Taking $y$ to be the canonical basis element $E_{12}-E_{21}$
gives another  example of a 3 dimensional Hilbertian Jordan operator algebra.  
If  one
wants a unital example then simply add scalar multiples of $I_3$.  

 Similarly we obtain interesting examples  of Jordan operator algebra structures on the $4 \times 4$ antisymmetric matrices.
These JC*-triples are not Hermitian, and need not be JW*- or
JB*-algebras.
 
A sometimes useful example (see \cite{BWj2})
is the Hilbertian operator space $\Phi_n$ (or $\Phi(I)$,
see \cite[Section 9.3]{Pisbk}), the span of operators $(T_k)$ (with $k \leq n$ or $k \in I$)
 satisfying the
CAR (canonical anticommutator relations).  This is a Jordan
operator algebra with zero product, and is a JC*-triple, so that this
example falls into the above discussion with $y = 0$.

The infinite dimensional 
separable
Hilbertian JC$^*$-triples are of four types \cite{N2}.  
The fourth type is 
the most interesting: namely the space of creation operators on the antisymmetric Fock space. 
Fixing a contractive element $y$ and considering the multiplication $m_y$ above gives an 
interesting Hilbertian Jordan operator algebra.   Again if  one
wants a unital example then simply unitize.

Another of the four types mentioned in the last paragraph gives a particularly simple 
example which 
we will use to illustrate the principles in the first paragraph of Example \ref{jajcst}.
This   is the operator space $R \cap C$,
which may be viewed as the subspace $W$ of $B(l^2 \oplus l^2)$ consisting of
 infinite matrices of block $2 \times 2$ matrix form
$$\left[ \begin{array}{cccccccc}
z_1 & z_2 & z_3 & \ldots & |  & 0 & 0 & \ldots \\
0 &   0   & 0   & \ldots      & |  & 0 & 0 & \ldots \\
0 &   0   & 0   & \ldots      & |  & 0 & 0 & \ldots \\
\vdots & \vdots & \vdots &  \ldots &  &  \vdots & \vdots &  \ldots \\
- & - & - & \ldots &  & - & - & \ldots \\
0 & 0 & 0 & \ldots & |  & z_1 & 0 &   \ldots \\
0 & 0 & 0 & \ldots & |  & z_2 & 0 &   \ldots \\
\vdots & \vdots & \vdots &  \ldots &   &  \vdots & \vdots &  \ldots 
\end{array} \right]. $$
Note that this subspace $W$ of $B(l^2 \oplus l^2)$ is not a
subalgebra, but is closed under squares, hence is a Jordan operator algebra, but it is not a JC$^*$-algebra since it
is not selfadjoint.
However it is also a JC$^*$-triple, and if we set $y = E_{11} \oplus E_{11} \in W$ then the multiplication 
$m_y$ defined above agrees with the natural Jordan product on $W$.   Note too that this particular example is
of the form $\{ (\theta(a), \pi(a)) \in C  \oplus^\infty D \}$ 
for a homomorphism $\theta$ and antihomomorphism $\pi$ described at the start of this section.

Many of the examples above are isometric as Banach spaces
and isomorphic as  Jordan algebras, but very different completely 
isometrically, that is as operator spaces in the sense of \cite{Pisbk,BLM}.

There are more than four types of contractively complemented Hilbertian JC$^*$-triples   in finite dimensions \cite{NR05}, for example the spaces $H^k_n$ studied in \cite{NR05}.   
For specificity, an interesting explicit rectangular 4 dimensional example is given in on p.\ 2260 of \cite{NR03}.   
 \end{example}

\begin{example}   Many of the examples above are unital Jordan operator algebras.    Examples of approximately unital Jordan operator algebras are easily manufactured from real positive (accretive) operators.   We mention two methods 
to do this; in fact every approximately unital Jordan operator algebra may be built in both of these ways.   First, if $A$ is any Jordan  operator algebra (say, a $C^*$-algebra), and $E$ is any nonempty convex set 
of real positive (accretive) elements in $A$, then   $\overline{\{ x A x : x \in E \}}$ 
(which equals $\overline{\{ x a y + ya x : x, y \in E , a \in A \}}$)  is an approximately unital  Jordan operator algebra.   See  Theorem 3.18 (2) in \cite{BWj}.    Note that if $A$ is an  approximately unital  Jordan operator algebra and we take
$E = {\mathfrak r}_A$ then we recover $A$ by this construction (see the proof of \cite[Corollary 4.2]{BWj}). 
 We also recall that if in addition $A$ is separable 
(but not necessarily assumed approximately unital)
then we can take $E$ above to be a singleton set
by \cite[Theorem 3.20]{BWj}.

A second method:  if $A$ is as above then by Proposition 4.4 in \cite{BWj}, the smallest norm closed Jordan algebra of $A$ generated by any set of real positive elements in $A$ 
 is an approximately unital  Jordan operator algebra. \end{example}

\begin{example}  We give 
an example of a Jordan operator algebra $A$ which is not 
an associative operator algebra, but is  completely isometrically
Jordan isomorphic to an associative operator algebra.   More generally, 
  a unital or approximately unital Jordan operator algebra can generate many quite different
associative operator algebras, depending on where the Jordan operator algebra is represented 
(completely isometrically) as a Jordan subalgebra.  One 
example to illustrate this is found in
\cite[Example 6.4]{BNp}, where a finite dimensional unital associative operator algebra
$B$ is constructed together with a unital completely isometric 
Jordan homomorphism of $B$ onto another 
unital Jordan operator algebra $A = T(B)$ which is not an associative operator algebra.
More explicitly, $A = \Phi_2$ (see the end of Example \ref{jajcst}) is
completely isometrically
Jordan isomorphic to $B = \Phi_2$ with zero operator algebra product.
Note that here 
the associative operator algebra generated by $A$, which is 
also finite dimensional (contained in $M_n$ for some $n$) will have strictly
larger dimension than $B$.      \end{example}

It is easy to see that  if $A$ is a  Jordan operator algebra then so is  $A^{{\rm op}}$ and $A^*$ (using the notation of  2.2.8 in \cite{BLM})
with the same Jordan product, and so is 
$A^{{\rm symm}} = \{ (a , a^\circ) \in A \oplus^\infty A^{{\rm op}} : a \in A \}$.

\section[Jordan variants of Hay and Brown's results]{Jordan variants of Hay's theorem and some variants of results of Brown} \label{Jcompl}

In this section we prove the hard technical facts that comprise the main breakthrough of the present paper.   We begin with a relatively easy 
two-sided analogue of   \cite[Proposition 3.1]{Hay}.   

\begin{proposition}\label{gamlemma1} Let $X$ be a closed subspace of a
  $C^*$-algebra  $B$, and let $q \in B^{**}$ be a closed projection such that $X^{\perp}  \subset
(qXq)_{\perp}$.
That is if $\psi \in B^*$ annihilates $X$ then
  $\psi$ is in the annihilator of $qXq \subset B^{**}$ in $B^*$.  
Let $I = \{x \in X : qx q= 0 \}$.  Then $qXq$ is completely isometric
  to $X/I$ via the map $x+I \mapsto qxq$. 
\end{proposition}

\begin{proof}  We will be a bit sketchy in places, since we are following the idea in \cite[Proposition 3.1]{Hay},
which in turn follows the idea in \cite[Lemma II.12.3]{Gam}.   On the other hand we will add more details in the computations
of the matrix norms to display that the new situation does not introduce any problems at the matrix 
norm level.  The verbosity is also justified because this proof needs to serve as a template for the more complicated later proof  
Theorem \ref{gamlemma2}.  

Thus $I$ is the kernel of the completely contractive map $x
  \mapsto qxq$ on $X$, so
  that this map factors through the quotient $X/I$:
  $$X \overset{Q}{\rightarrow} X/I \rightarrow qXq,$$
  where $Q$ is the natural quotient map.  Taking adjoints, we find that the associated map from $(qXq)^*$ to $X^*$, is given by
  $\varphi \mapsto \varphi(q \cdot q),$ for each $\varphi \in (qXq)^*$.
  Identifying $(qXq)^*$ with $(qBq)^*/(qXq)^{\perp}$ and $X^*$ with
 $B^*/X^\perp$, the map above  takes an element $\varphi +
  (qXq)^\perp \in (qBq)^*/(qXq)^{\perp}$ to the element $\varphi(q \cdot q) + X^\perp
\in B^*/X^\perp$. 
As   in \cite[Proposition 3.1]{Hay}, it suffices to show
  that  the map $\varphi +
  (qXq)^\perp \to \varphi(q \cdot q) + X^\perp$, which is fairly clearly a complete contraction, is
  a complete isometry. 
  So let $[\varphi_{ij}] \in M_n((qBq)^*)$ and let $[\psi_{ij}] \in M_n(X^\perp)$.     
By  hypothesis then, $[\psi_{ij}] \in M_n((qXq)_{\perp})$, so that if $\widetilde{\psi_{ij}}$ denotes the extension of $\psi_{ij}$ to a weak* continuous
map on $B^{**}$, we have $[\widetilde{\psi_{ij}} |_{qBq}] \in M_n((qXq)^\perp )$.  Now,
  $$\|[\varphi_{ij} + 
\widetilde{\psi_{ij}}|_{qBq}]\|_n = \sup \{ \| [\varphi_{ij}(qb_{kl}q)+  \widetilde{\psi_{ij}}(qb_{kl}q) ] \|
: [qb_{kl}q] \in {\rm Ball}( M_m(qBq)) \} .$$
The left ideal $L$ supported by $q^\perp$ is well known (see e.g.\
p.\ 917 in \cite{Brown}, or \cite{Kirch}) to have the property that 
$(L + L^*)^{\perp \perp} = B^{**} q^\perp + q^\perp B^{**}$, the latter sum being 
weak* closed. 
Hence $$\{ b \in B : q b q = 0 \} = B \cap (B^{**} q^\perp + q^\perp B^{**}) = L+L^* .$$
   The canonical map $B/ \{ b \in B : q b q = 0 \} \to qBq$  is a complete isometry 
since the bidual of $B/ \{ b \in B : q b q = 0 \}$ is $B^{**}/(B^{**} q^\perp + q^\perp B^{**}) 
\cong q B^{**} q$ (the latter since the map $\eta \to q \eta q$ from $B^{**}$ to $q B^{**} q$
 is a complete quotient map with kernel $B^{**} q^\perp + q^\perp B^{**}$. 
Moreover, $L + L^*$ is proximinal in $B$ (this follows from \cite[Theorem 3.3 (b)]{Brown},
see the slightly misstated Remark after 3.4 there), and thus 
$M_m(L) + M_m(L)^*$ is proximinal in $M_n(B)$.   Hence given $b_{kl} \in B$ with $[q b_{kl} q ] \in 
{\rm Ball}(M_m(qBq))$, or equivalently with
$\| [b_{kl}] + M_m(L + L^*) \| \leq 1$, there exists 
$[a_{kl}] \in M_m(L + L^*)$ with $\| [b_{kl}] + [a_{kl}]  \| \leq 1$.
Since $q a_{kl} q = 0$, we may replace $b_{kl}$ by $b_{kl} + a_{kl}$ in the sup formula 
above to deduce that 
 $$\|[\varphi_{ij} +
\widetilde{\psi_{ij}}]\|_n = \sup \{ \| [\varphi_{ij}(qb_{kl}q)+  \widetilde{\psi_{ij}}(qb_{kl}q) ] \|
: [b_{kl}] \in {\rm Ball}( M_m(B)) \} .$$

Continuing to follow \cite[Proposition 3.1]{Hay}, if $(e_t)$ is a net 
of contractions in $B$ with weak* limit $q$, and if $[b_{kl}] \in {\rm Ball}( M_m(B))$,
 then we have
$$\| [\varphi_{ij}(qe_t b_{kl} e_s q)+  \widetilde{\psi_{ij}}(e_t b_{kl} e_s) ] \|
\leq \| [\varphi_{ij}(q \cdot q) + \psi_{ij} ] \|_{M_n(B^*)} .$$ 
Taking a double weak* limit with $s$ and $t$ 
we see that 
$\| [\varphi_{ij}(qb_{kl}q)+  \widetilde{\psi_{ij}}(qb_{kl}q) ] \|
\leq \| [\varphi_{ij}(q \cdot q) + \psi_{ij} ] \|_{M_n(B^*)}$.
Thus $\|[\varphi_{ij} +
\widetilde{\psi_{ij}} |_{qBq} ]\|_n \leq \| [\varphi_{ij}(q \cdot q) + \psi_{ij} ] \|_{M_n(B^*)}$, giving $\|[\varphi_{ij} + (qXq)^\perp ] \| \leq \| [\varphi_{ij}(q \cdot q) + \psi_{ij} ] \|_{M_n(B^*)}$.
Taking the infimum, $$\|[\varphi_{ij} + (qXq)^\perp ] \|
\leq \| [\varphi_{ij}(q \cdot q) + X^\perp] \| .$$ 
This is the desired inequality establishing that 
the map $(qBq)^*/(qXq)^\perp \to B^*/X^\perp : \varphi + (qXq)^\perp \mapsto 
\varphi(q \cdot q) + X^\perp$ is a complete isometry.
\end{proof}

One of the conditions in the previous result simplifies if $A$ is a Jordan subalgebra:

\begin{lemma} \label{tfqin}  Let $A$ be a closed Jordan subalgebra of a unital $C^*$-algebra $B$, and let $q$ be a projection in $A^{\perp \perp}$.
Then $q \in A^{\perp \perp}$ if and only if $A^{\perp} \subset (qAq)_{\perp}$.  
\end{lemma}

\begin{proof}  Suppose that $q \in A^{\perp \perp}$, and that $\varphi \in A^\perp$
and $(a_t)$ is a net in $A$ with weak* limit $q$.
Then $$\varphi(qxq) = \lim_t \lim_s \, \frac{1}{2} \varphi(a_s x a_t +
a_t x a_s) = 0, \qquad x \in A, $$ 
since $\frac{1}{2} (a_s x a_t +
a_t x a_s) \to qxq$ weak* in $B^{**}$ with $s, t$.   

Conversely, suppose that $A^{\perp} \subset (qAq)_{\perp}$ and $\psi \in A^{\perp}$.
 Then $$0 = \langle q a_t q , \psi \rangle \to \langle q  , \psi \rangle ,$$
since $q a_t q \to q^3 = q$ weak*.
 So $q \in A^{\perp \perp}$.  
\end{proof}

\begin{lemma}  \label{distH} Let $A$ be a closed Jordan subalgebra of a $C^*$-algebra $B$,
 let $p$ be an $A$-open projection in $A^{**}$, and let  $D = A \cap
p A^{**} p$
be the HSA in $A$ supported by $p$.   If $x \in A$ with
$[p,x,p^\perp] = 0$ then the distance $d(x, D)$ from $x$ to $D$ equals
$\| (1-p)x (1-p) \|$. \end{lemma}  

\begin{proof}    Let $q = 1-p \in (B^1)^{**}$, where $1$ is the identity of $B^1$.    Note that 
$A^1$ is a closed Jordan subalgebra of $B^1$.      By the line above Corollary 3.1 in \cite{BWj} and the line below
Proposition 3.2 there 
we have that $p$ is $A^1$-open and is open in $B^1$.  Then $q$ is closed in $(B^1)^{**}$.
Let $X = \{ x \in A^1 : [q,x,p] = 0 \}$.   This is a unital Jordan
algebra, since $X = A^1 \cap (q (A^1)^{**} q + q^\perp (A^1)^{**} q^\perp)$,
which is clearly closed under squares.  Moreover $q \in X^{\perp \perp}$
since 
$$p = q^\perp \in (A \cap p (A^{1})^{**} p)^{\perp \perp}
\subset  (A^1 \cap q^\perp (A^1)^{**} q^\perp)^{\perp \perp}
\subset X^{\perp \perp}$$
by the definition of $A$-open projection.
So by Proposition \ref{gamlemma1},  $qXq$ is isometric to $X/I$ via the
map $x+I \mapsto qxq$, where $$I = \{ x \in X : qxq = 0 \} =
\{ x \in A^1 \cap q^\perp (A^1)^{**} q^\perp \} = D.$$    Since $d(x, D) = \| x + I \|$,
we are done.  \end{proof} 

\begin{lemma} \label{lemma2}  
Let $X$ be a closed subspace of a unital $C^*$-algebra 
$B$ and let $q$ be a  closed projection in $B^{**}$ such that 
$X^{\perp} \subset (qXq)_{\perp}$ as in Lemma {\rm 
\ref{gamlemma1}}.  Suppose that $h \in B$ is strictly positive and commutes 
with $q$, and suppose that $a \in X$ with  $a^* q a \leq h$. Then given 
$\epsilon >0$ there exists  $b \in X$ such that $qbq = qaq$ and $b^{\ast}b
 \leq h + \epsilon 1$. \end{lemma}

\begin{proof}  Again we will be sketchy in places, since we are following
\cite[Proposition 3.2]{Hay} and 
\cite[Lemma II.12.4]{Gam}.  First suppose that $h=1$.
Suppose that $a \in X$ and $a^{\ast}qa \leq 1$.  
Hence $\|qaq \| \le \|qa \| \le 1$. 
Let $I = \{x \in X:qxq = 0\}$.
As in the just cited proofs, but using Lemma {\rm
\ref{gamlemma1}} and replacing several of the $qa$ and $qb$ in 
\cite[Proposition 3.2]{Hay} by $qaq$ and $qbq$, 
there exists $b \in X$ with $qaq = qbq$
and $b^* b \leq (1+\epsilon) 1 = h + \epsilon 1$. 

In the general case we have 
$h^{-1/2} a^* q a h^{-1/2} \leq 1$.  If $\psi \in (X h^{-1/2})^{\perp}$, 
then $\psi(\cdot h^{-1/2}) \in X^{\perp}$ and thus $\psi(\cdot h^{-1/2}) \in (qXq)_{\perp}$ by hypothesis. Hence $\psi$ lies in $(qX h^{-1/2}q)_{\perp}$. 
Thus by the $d = 1$ case, there exists  $b \in X$ such that $qbq h^{-1/2} = qaq
h^{-1/2}$ and $h^{-1/2} b^{\ast}b h^{-1/2}  \leq (1+\frac{\epsilon}{\| p \|}) 1$.
Hence $b^{\ast}b \leq h + \epsilon 1$.
\end{proof}
 
We will use the case of the previous lemma where $1_B \in X$ and $a = 1_B$.

\begin{lemma} \label{lemma3}
Let $A$ be a closed Jordan subalgebra of a unital $C^*$-algebra $B$ with $1_B \in A$,
and let $q$ be a closed projection in $B^{**}$ such that
$q \in A^{\perp \perp}$.
 Then there exists a doubly indexed net $(b^t_n)$ in $A$ 
(with $n \in \Ndb$)  such that $qb^t_{n}q = q$, $\| b^t_{n} \|^2 \le 1+\frac{1}{n}$, and $b^t_{n} \rightarrow q$ in the weak* topology with $t$ and $n \to \infty$. 
Furthermore $\| qb^t_{n}(1-q) + (1-q)b^t_{n}q \| \le \frac{1}{\sqrt{n}}$.  \end{lemma}

\begin{proof}  If $(x_t)$ is an increasing cai for the hereditary subalgebra defined by $q^\perp$, then as in the start of the proof of \cite[Theorem 2.2]{Bnpi} 
$f^s = \frac{1}{m} 1 + \frac{m-1}{m} (1-x_t)$ defines a net of strictly 
positive elements norm $1$ elements of $B$ which commute with and dominate $q$,  and 
satisfy $f^s q = q f^s = q$, and $f^s \to q$ weak*.
Since $q \in A^{\perp \perp}$ we have  $A^\perp \subset 
(qAq)_\perp$ by Lemma \ref{tfqin}.    
By Lemma \ref{lemma2}  for each $n \in \Ndb$ there exists $b^s_n \in A$ such that 
$q b^s_n q = q$ and $(b^s_n)^* b^s_n \leq f^s + \frac{1}{n}$.
Hence $\| b^s_n \|^2 \leq 1 + \frac{1}{n}$.
Writing $b_n$ for $b^s_n$ we have \begin{eqnarray*}
q + ((1-q) b_{n}q)^{\ast}(1-q)b_{n}q &=& qb^{\ast}_{n}qb_{n}q +
 q b^{\ast}_{n}(1-q)(1-q)b_{n}q \\
&= & qb_{n}^{\ast}b_{n}q\\ 
& \le & q + \frac{1}{n} . \end{eqnarray*} 
Hence $\| (1-q) b^s_{n}q \|^2 = \| ((1-q) b^s_{n}q)^{\ast}(1-q)b^s_{n}q \|
\leq \frac{1}{n}$.
Similarly,
$$ q + qb_{n}(1-q)( qb_{n}(1-q))^{\ast} = q b_{n} b_{n}^* q 
\leq (1 + \frac{1}{n}) q ,$$
and so $\| qb^s_{n}(1-q) \|^2 \leq \frac{1}{n}$. 
We have $$\| qb^t_{n}(1-q) + (1-q)b^t_{n}q \| = \max \{
\| qb^t_{n}(1-q) \| , \| (1-q)b^t_{n}q \| \} \leq \frac{1}{\sqrt{n}} .$$   Note that 
$$q^\perp (b^s_{n})^* b^s_n q^\perp \leq 
q^\perp (f^s + \frac{1}{n}) q^\perp \to 0$$
weak* with $s$ and $n \to \infty$.  In the universal 
representation of $B$ on a Hilbert space $H$ we have  
$$\| b^s_n q^\perp \xi \|^2 = \langle q^\perp (b^s_{n})^* b^s_n q^\perp \xi , \xi \rangle  \to 0 , \qquad \xi \in H .$$
Thus $b^s_n q^\perp \to 0$ strongly, and hence $q^\perp b^s_n q^\perp \to 0$ weak*.  
It follows that $b^s_n = q + [q,b^s_{n},q^\perp] + q^\perp b^s_n q^\perp \to q$ 
weak*.  \end{proof} 

\begin{corollary} \label{weakHay} Let $A$ be a closed Jordan subalgebra of
a $C^{\ast}$-algebra $B$ with $1_B \in A$, and let $q$ be a  closed projection in $B^{**}$ such that
$q \in A^{\perp \perp}$.
 Then there exists a bounded doubly indexed net $(a^t_n)$ in $A$
(with $n \in \Ndb$)  such that $q a^t_{n}q = 0$, 
$\| q a^t_{n}(1-q) + (1-q) a^t_{n}q \| \le \frac{1}{\sqrt{n}}$, and $a^t_{n} \rightarrow q^\perp$ in the weak* topology with $t$ and $n \to \infty$.
\end{corollary}

\begin{proof}  
Just let $a^t_{n} = 1 - b^t_{n}$ from the previous corollary.
 \end{proof}

\noindent We will improve
the previous corollary later.
For now we turn to 
some lemmas about proximinality in $C^{\ast}$-algebras.

 A hereditary subalgebra $D$ in a $C^*$-algebra $B$ need not be proximinal, as Brown showed in 3.12 in 
\cite{Brown}.  However we will prove that $D$ is proximinal in an important part of $B$, namely 
$L + L^* = \{ b \in B : qbq = 0 \}$, where $L = \overline{AD}$ is the left ideal associated with $D$, and $q^\perp$ is the support projection of $D$ in $B^{**}$.   That is, $L \cap L^*$ is proximinal in $L + L^*$.  
To do this we will use some of Brown's ideas for his proof that
$L +R$ is proximinal in $B$, together with some new ingredients.  Here $L$ (resp.\ $R$) is a closed  left 
(resp.\ right) ideal in $B$.     We recall that the precise distance $d(x, D) = \max \{ \| xq \|, \| qx \| \}$
(see \cite[p.\ 920]{Brown}), but 3.12 in \cite{Brown} shows that this distance is not always achieved.
Nonetheless our result shows that  this distance is  achieved if also $qx q = 0$.

We first adapt one of Brown's $W^*$-algebra results (\cite[Lemma 3.1]{Brown}) . 

\begin{lemma} \label{Brw}   Let $q$ be a projection in a $W^{\ast}$-algebra $M$. 
Let $\epsilon >0$ and suppose $x \in M$ with $qxq=0$. If $\| (1-q)xq +q x (1-q) \| \le 1$ and $\|x \| \le 1+ \epsilon$, 
then there exists  $y \in (1-q)M(1-q)$ such that $\|y \| \le 3\sqrt{2\epsilon + \epsilon^{2}}$ and $\|x-y \| \le 1$.  
 \end{lemma}
 
\begin{proof}  Since $qxq = 0$, $xx^{\ast} = (1-q)xqx^{\ast}(1-q) + x(1-q)x^{\ast} \le 1+2\epsilon + \epsilon^{2}$.
 Hence, 
\[
xq^\perp (x q^\perp)^{\ast} \le 1+2\epsilon + \epsilon^{2} - q^\perp xqx^{\ast} q^\perp . \] 
By the well known `majorization-factorization' principle in 
von Neumann algebras, it follows that 
\begin{eqnarray}\label{1}
x q^\perp = \sqrt{1+2\epsilon + \epsilon^{2} - q^\perp xqx^{\ast} q^\perp} \, t
\end{eqnarray} 
for some $t \in {\rm Ball}(M)$.
We may assume that $t = t q^\perp$.  Now let $z = t \, \sqrt{1- q^\perp xqx^{\ast}q^\perp}$, 
which exists since $\| q^\perp xq \| \le 1$ (note that 
$\|q^\perp xq +q x q^\perp \|  = \max \{ \| q^\perp xq \|, \|qx q^\perp \| \} \leq 1$). 
 It is clear that $\| x q^\perp - z \| \le \sqrt{2\epsilon + \epsilon^{2}}$ by functional calculus. Also, letting $w =x - (xq^\perp - z) =q^\perp xq + z q^\perp$, we have
\begin{eqnarray*}
ww^{\ast} & = & q^\perp xqx^{\ast} q^\perp  + z q^\perp z^{\ast} \\
                 &=& q^\perp x q x^{\ast} q^\perp + \sqrt{1- q^\perp xq x^{\ast}
q^\perp)}tt^{\ast}\sqrt{1- q^\perp xqx^{\ast} q^\perp)} \\
                 &\le& 1 .
 \end{eqnarray*}
Thus, $\|w \| \le 1$. Now from using a binomial series for the root in
equation (\ref{1}), and similarly for the root defining $z$,   we have
\begin{eqnarray*}
qx q^\perp  = (1+\epsilon)pt q^\perp = (1 + \epsilon)qz q^\perp = (1+\epsilon)qw q^\perp \end{eqnarray*}
and clearly $q^\perp xq =  q^\perp wq$ and $qwq = 0$. To recap, we have added an element $z - x q^\perp$ of norm $\sqrt{2\epsilon + \epsilon^{2}}$ to $x$ to obtain an element $w$ such that: 
\begin{enumerate}
\item $\|w \| \le 1$,
 \item $qwq =0$.
 \item $q^\perp xq = q^\perp wq$,
 \item $qx q^\perp = (1+\epsilon)qw q^\perp$,
\item $\| x- w \| \leq \sqrt{2\epsilon + \epsilon^{2}}$. \end{enumerate} 
Now consider $(1+\epsilon)w$, which, like $x$, has norm $\leq 1+\epsilon$ yet $\|q((1+\epsilon)w q^\perp \|= \|qx q^\perp \| \le 1 $.  By a symmetric argument to the one that produced the enumerated list above, there is an element $u = q^\perp u$ such that $\|u\| \le \sqrt{2\epsilon + \epsilon^{2}}$ and, letting $w^{\prime} = (1+\epsilon)w - u$, 
\begin{enumerate}
\item $\| w' \| \le 1$ and $\| (1+\epsilon)w - w^{\prime} \| \le \sqrt{2\epsilon + \epsilon^{2}}$,
 \item $qw^{\prime}q = 0$,
 \item $qw^{\prime} q^\perp = (1+\epsilon)qw q^\perp = qx q^\perp$,
\item $q^\perp w^{\prime}q = q^\perp xq$,
  \end{enumerate} 
\noindent noting for (4), as before, that $q^\perp w^{\prime}q(1+ \epsilon) =
q^\perp w(1+\epsilon)q = (1+\epsilon) q^\perp xq$.  

Finally, setting  $y = x - w^{\prime}$, we have $qyq = 0,$
and $[q,y q^\perp] = 0$ from (3) and (4).  So $y \in q^\perp M q^\perp$.
  Also  $\| x- y \| \le 1$, and  furthermore,
\begin{eqnarray*}
\| y \| = \| x - w^{\prime} \|  \le \|x - w \| + \|\epsilon w \| + \|(1+\epsilon)w - w^{\prime} \| \le 3\sqrt{2\epsilon + \epsilon^{2}} , \end{eqnarray*}
where we have used (1) and (5) in the earlier list, 
and (1) in the last list, in the last inequalities.  \end{proof}

\begin{lemma}\label{Br32prox}
Let $q$ be a closed projection in $B^{\ast\ast}$ for a  $C^{\ast}$-algebra $B$.
  If $\epsilon >0$ and $x \in B$ with $qxq = 0$ and $\| q x q^\perp 
+ q^\perp x q \| \le 1$ and $\|x\| \le 1 + \epsilon$,
 then for all $\delta > 0$ there exists  $y \in B$  such that 
$$y = q^\perp y q^\perp, \; \|y \| \le 3\sqrt{2\epsilon +\epsilon^{2}}, \; \textrm{and} \;  \|x-y \| < 1 + \delta.$$  \end{lemma}
\begin{proof}  This is almost exactly like the proof of Lemma 3.2 in \cite{Brown},
but replacing $A$ there with our  $B$, $L+R$ there by 
$D = q^\perp B^{**} q^\perp \cap B$, $\sqrt{2\epsilon +\epsilon^{2}}$ there 
by $3\sqrt{2\epsilon +\epsilon^{2}}$, and relying on facts established in Lemma \ref{Brw}.
Note  $D$ is the HSA in $B$ supported 
by $q^\perp$. 
Also we replace  $C$ there with $\{ y \in D : \| y \| \leq
3\sqrt{2\epsilon +\epsilon^{2}} \}$.    We have $D^{\perp \perp} = q^\perp B^{**} q^\perp$, 
and so the weak* closure of $C$ is $\{ \eta \in q^\perp B^{**} q^\perp
: \| \eta \| \leq 3\sqrt{2\epsilon +\epsilon^{2}} \}$.   This is what 
is needed to adjust the proof  of \cite[Lemma 3.2]{Brown} to our
setting, the other
ideas are the same.
 \end{proof}

\begin{lemma}\label{Br33prox}
Let $q$ be a closed projection in $B^{\ast\ast}$ for a  $C^{\ast}$-algebra $B$. 
If $\epsilon >0$  and $x \in B$ with $qxq = 0$ and $\| q x q^\perp + q^\perp x q \| \le 1$, and if $\|x\| \le 1 + \epsilon$, then for all $\epsilon^{\prime} > 3\sqrt{2\epsilon +\epsilon^{2}}$ there exists  $y \in B$  such that 
$$y = q^\perp y q^\perp, \;
 \|y \| \le \epsilon^{\prime}, \; \textrm{and} \, \|x-y \| \le 1 .$$  \end{lemma}
\begin{proof}  Exactly like the proof of Theorem 3.3 of \cite{Brown} but using Lemma \ref{Br32prox} instead
of his variant of that result,
and replacing $L+R$ by $D$ as we did in the  proof of Lemma \ref{Br32prox}, and $\sqrt{2\epsilon +\epsilon^{2}}$ there  
by $3\sqrt{2\epsilon +\epsilon^{2}}$.
 \end{proof}

We now unravel the meaning and implications of Lemma \ref{Br33prox}.   

If $B$ is a $C^*$-algebra and $q$ is a closed
projection in $B^{**}$  
recall that 
$$\{ b \in B : qbq = 0 \} = B \cap 
(B^{**}_0(q) + B^{**}_1(q)) = 
L + L^*,$$  where $L$ is the closed left ideal supported by $q^\perp$.

\begin{proposition} \label{cl01d}   If $B$ is a $C^*$-algebra and $q$ is a closed
projection in $B^{**}$ supporting a closed left ideal $L$ in $B$, then 
$(L+L^*)/D \cong   I_B = \{ [q,b,q^\perp] : b \in L + L^* \}$
 completely isometrically, where $D = B^{**}_0(q) \cap B = L \cap L^*$ is the HSA supported by $q^\perp$.
Also,  $(L+L^*)^{**} /B^{**}_0(q) \cong 
B^{**}_1(q)$ completely isometrically.   
\end{proposition}

\begin{proof}  As we said in the proof of Lemma \ref{gamlemma1}, 
$(L + L^*)^{\perp \perp} = B^{**} q^\perp + q^\perp B^{**}$, the latter sum being 
weak* closed and equaling $B^{**}_0(q) + B^{**}_1(q) = \{ \eta \in B^{**} : q \eta q = 0 \}$.
Also  $$L+L^* = \{ b \in B : q b q = 0 \} = B \cap (B^{**} q^\perp + q^\perp B^{**})$$ is
weak* dense in the space in the last line.   

The canonical map $P_1 : B^{**}_0(q) + B^{**}_1(q) \to B^{**}_1(q): \eta \to [q ,\eta ,q^\perp]$
is a complete contraction with kernel $B^{**}_0(q)$.   Thus we obtain a complete isometry
$$(L + L^*)^{**} /B^{**}_0(q) \cong 
B^{**}_1(q)$$  via the map $b + D \mapsto [q,b,q^\perp]$, which 
maps onto $\{ [q,b,q^\perp] : b \in L + L^* \}$.  
\end{proof}

Thus the quantity  $\| q x q^\perp + q^\perp x q \|$ in Lemma \ref{Br33prox}
is exactly the distance of $x \in L + L^*$ from the hereditary subalgebra
$D$ in $B$ which is supported by $q^\perp$.  With this in mind, 
Lemma \ref{Br33prox} yields what seems to be
a new purely $C^*$-algebraic result:  

\begin{theorem} \label{prox}  A hereditary subalgebra $D$ in a (possibly nonunital)  
$C^*$-algebra $B$ is
 proximinal in
$L + L^*$, where $L = \overline{AD}$ is the left ideal associated with $D$.
\end{theorem}

\begin{proof}     Let  $q^\perp$ be the support projection of $D$ in $B^{**}$.
Then $L + L^* = \{ b \in B : qbq = 0 \}$ as we have said before. 
Lemma \ref{Br33prox} may  then be  translated as saying that that there
is a function $f(\epsilon) \to 0$ as $\epsilon \to 0^+$, such that 
if $\epsilon >0$  and $x \in L + L^*$ with $d(x,D) \le 1$, and if 
$\|x\| \le 1 + \epsilon$, then there exists  $y \in D$  with
$\|y \| \le f(\epsilon)$ and $\|x-y \| \le 1$.
This kind of condition (which is also used in the Remark
after 3.4 of \cite{Brown}) is much stronger than some of the 
known `ball properties' that characterize
proximinality (see e.g.\ \cite[Proposition II.1.1]{HWW}).
\end{proof}

The following is a second, and much deeper, two-sided analogue of   \cite[Proposition 3.1]{Hay} (cf.\ Proposition \ref{gamlemma1}).   The case that $X$ is a HSA in $B$, and $q$ the complement 
of its support projection, is particularly interesting and will be used later.

\begin{theorem}\label{gamlemma2}   Let $X$ be a closed subspace of a
  unital $C^*$-algebra
  $B$.  Let $q \in B^{**}$ be a closed projection such that $qXq = \{0\}$ and such that if $\psi \in B^*$  annihilates
$X$,
then $\psi$ annihilates
$I_X$, where $I_X =  \{ [q,x, q^\perp] \in B^{**} : x \in X \}$ (that is, $\psi \in (I_X)_\perp$).
Let $J_X = \{ x \in X : [q,x, q^\perp]  = 0 \}$.   Then 
$X/J_X \cong I_X$ completely isometrically  via the map $x+J_X \mapsto [q,x,q^\perp]$. 
\end{theorem}

\begin{proof}     Since $qXq = (0)$ we have that $X \subset L+L^* \subset B$  where $L$ is the closed left ideal supported by $q^\perp$, 
and $I_X \subset B_1^{**}(q) \subset (L+L^*)^{\perp \perp}$.   It follows that 
saying that if $\psi \in B^*$  annihilates $X$ then $\psi$ annihilates
$I_X$, is equivalent to saying that if $\psi \in (L+L^*)^*$  annihilates
$X$ then $\psi$ annihilates $I_X$.   We will work with the latter statement below.  

We will be a bit sketchy in the places where we are following the ideas in \cite[Proposition 3.1]{Hay}
and Lemma \ref{gamlemma1} 
in a straightforward way. As in those proofs, $J_X$ is the  kernel of the completely contractive map $r_X : x
  \mapsto  [q,x, q^\perp]$ on $X$, so
  that this map factors through the quotient $X/J_X$:
  $$X \overset{Q}{\rightarrow} X/J_X \rightarrow I_X,$$
  where $Q$ is the natural quotient map.  Taking adjoints, we find that the associated map from $I_X^*$ to $X^*$, is given by
  $\varphi \mapsto \varphi \circ r_X,$ for each $\varphi \in I_X^*$.  Write $I_B =  \{ [q,x, q^\perp] \in B^{**} : x \in L + L^* \}$, and let $r : L+L^* \to I_B \subset B^{**}$ be the map $r(b) = P_1(b) = [q,\hat{b}, q^\perp]$. 
  Identifying $I_X^*$ with $I_B^*/I_X^{\perp}$ and $X^*$ with
 $(L+L^*)^*/X^\perp$, the map above  takes an element $\varphi +
  I_X^\perp$ to the element $\varphi \circ r + X^\perp$, for $\varphi \in I_B^*$.
As    before, it suffices to show
  that  $\varphi +
  I_X^\perp \to \varphi \circ r + X^\perp$ is
  a complete isometry from $I_B^*/I_X^{\perp}$ to $(L+L^*)^*/X^\perp$, where $\varphi \in I_B^*$.    
We
recall from Proposition \ref{cl01d}  that $(L+L^*)/D \cong   I_B =  \{ [q,b,q^\perp] : b \in L+L^* \}$,  where $D =  B^{**}_0(q) \cap B$  is the HSA supported by $q^\perp$.  Recall that $D$ is proximinal in $L+L^*$ by 
Theorem  \ref{prox}.     In $M_n(B^{**}) \cong M_n(B)^{**}$ we consider $r = q \otimes I_n$, a closed projection.
With respect to this projection
$\{ [b_{ij}] \in M_n(B) : [qb_{ij} q] = 0 \} = M_n(L+L^*)$,
and Proposition \ref{cl01d} holds at the matricial level.   For example, $$M_n(L+L^*)/M_n(D) \cong   I_{M_n(B)} = \{ [[q,b_{ij},q^\perp]] :  [b_{ij}] \in M_n(L+L^*) \}.$$   This uses the fact that $r^\perp  M_n(B^{**})  r^\perp
\cap M_n(B) = M_n(D)$.  
By Theorem \ref{prox}  and the identifications above, the hereditary subalgebra $M_n(D)$  is 
 proximinal in  $M_n(L+L^*)$.
Hence if $\| [[q, b_{ij} , q^\perp]] \| \leq 1$ for some $[b_{ij}] \in M_n(L+L^*)$ 
then by proximinality there exists $[a_{ij}] \in M_n(D)$ such that 
$$\| [[q, b_{ij} , q^\perp]] \| = \| [ [q, b_{ij} +a_{ij} , q^\perp ]] \| = \|[  b_{ij}+D ] \| = \| [b_{ij}+a_{ij} ] \|.$$ 
If $\psi \in (L+L^*)^*$ annihilates $X$, then by hypothesis $\psi \in (I_X)_\perp$.  Let $\tilde{\psi}$ be the canonical 
weak* continuous extension of $\psi$ to a weak* continuous
functional on $(L+L^*)^{**} = B^{**}_1(q) + B^{**}_0(q)$.   By the last centered equation we see that
$\| [\varphi_{kl} + \widetilde{\psi_{kl}} |_{I_B} ] \|_{M_m(I_B^*)}$ equals 
 $$\sup \{ \| [ \varphi_{kl}(r(b_{ij})) + \widetilde{\psi_{kl}} (r(b_{ij})) ] \|
: [b_{ij}] \in {\rm Ball}(M_n(L+L^*) \} .$$    For simplicity we assume $m = n = 1$ henceforth; once we 
have completed that case it will be easy to 
see that the matricial case is similar.   We have shown that 
$\| \varphi  + \tilde{\psi} |_{I_B} \|_{(I_B)^*} \leq
\| \varphi \circ r + \tilde{\psi} \circ r  \|_{(L+L^*)^*}$.   
Claim: $$\| \varphi \circ r + \tilde{\psi} \circ r \|_{(L+L^*)^*} \leq
\| \varphi \circ r + \psi \|_{(L+L^*)^*}
.$$
If this were the case then $\| \varphi  + \widetilde{\psi} |_{I_B}  \|_{(I_B)^*} \leq
\| \varphi \circ r+ \psi \|_{(L+L^*)^*}$.
Thus $\| \varphi +I_X^\perp  \| \leq \| \varphi \circ r + \psi \|_{(L+L^*)^*}$.   Taking an infimum we get
$\| \varphi +I_X^\perp \| \leq \| \varphi \circ r+  X^{\perp} \|$, 
so that  $\varphi +   I_X^\perp \to \varphi \circ r + X^\perp$ is
  an isometry  as desired.   Similarly, it is a complete isometry.  

We now prove the Claim.
A variant of the argument below would give an alternative proof of the last steps of \cite[Proposition 3.1]{Hay}
and of Proposition \ref{gamlemma2}.
Let $\rho = \varphi \circ r$.
The claim then says that $\| \rho + \tilde{\psi} \circ r \|_{(L+L^*)^*} \leq
\| \rho + \psi \|_{(L+L^*)^*}$.   The canonical weak* continuous extension of $\rho$ to $(L+L^*)^{**}$ is 
$\tilde{\varphi} \circ r^{**}$, since $\tilde{\varphi} (r^{**}(\hat{b})) = \varphi(r(b)) = \rho (b)$ for $b \in B$.
We are viewing $r^{**}$ as a map into $(I_B)^{**} \subset B^{****}$.
The canonical weak* continuous extension of $\tilde{\psi} \circ r$ to $B^{**}$ is  
 $\tilde{\psi} \circ P_1$, since $\tilde{\psi}(P_1(\hat{b}) ) = \tilde{\psi}(r(b))$ for $b \in B$.
Thus
$$\| \rho + \tilde{\psi} \circ r \|_{(L+L^*)^*}  = \|(\tilde{\varphi} \circ r^{**} + \tilde{\psi} \circ P_1) |_{(L+L^*)} \|_{(L+L^*)^*} 
\leq  \| \tilde{\varphi} \circ r^{**} + \tilde{\psi} \circ P_1 \| .$$
We next show that $r^{**}  = r^{**}  \circ P_1$, or equivalently that $r^{**}$ annihilates
$qB^{**} q + q^\perp B^{**} q^\perp$.   
Writing $i_X$ for the canonical map of a space into its bidual, by weak* density  we must have 
$r^{**} = [i_{B^{**}}(q), (i_B)^{**}(\cdot) , i_{B^{**}}(q^\perp)]$, since both sides agree on $B$.
Indeed $[i_{B^{**}}(q), (i_B)^{**}(i_B(b)) , i_{B^{**}}(q^\perp)]$ equals
$$[i_{B^{**}}(q),i_{B^{**}}(i_B(b))  , i_{B^{**}}(q^\perp)] 
= i_{B^{**}}([q, i_B(b), q^\perp]) = r^{**}(i_B(b)), \qquad b \in B .$$
Since $r^{**} = [i_{B^{**}}(q), (i_B)^{**}(\cdot) , (i_{B^{**}}(q))^\perp]$ and 
$$(i_B)^{**}(q \eta q) = (i_B)^{**}(q) (i_B)^{**}(\eta) (i_B)^{**}(q),$$ 
it is now clear that $r^{**}$ annihilates $qB^{**} q$.   Similarly, it annihilates $q^\perp B^{**} q^\perp$.

We now have: 
$$\| \rho + \tilde{\psi} \circ r \|_{(L+L^*)^*}  \leq 
 \| (\tilde{\varphi} \circ r^{**} + \tilde{\psi}) \circ P_1 \| 
\leq  \| \tilde{\varphi} \circ r^{**} + \tilde{\psi}  \|
= \| \rho + \psi \| ,$$
as desired (we are using the fact that $\tilde{\varphi} \circ r^{**} + \tilde{\psi}$ is the 
canonical weak* continuous extension of $\rho + \psi$, so has the same norm).
\end{proof}  

\begin{corollary} \label{coh}    Let $X$ be a closed subspace of a
  unital $C^*$-algebra
  $B$.  Let $q \in B^{**}$ be a closed projection  such that $qXq = (0)$ and such that if $\psi \in B^*$
annihilates
$X$, 
then $\psi$ annihilates $\{ [q,x, q^\perp] \in B^{**} : x \in X \}$.
If $x \in X$ with $\| [q,x, q^\perp] \| \leq 1$, and if $\epsilon > 0$ then there is an element $y \in X$ with
$[q,x, q^\perp]  = [q, y , q^\perp]$ and $\| y \| \leq 1 + \epsilon$.   
\end{corollary}

\begin{proof}  By  Theorem \ref{gamlemma2},  $X/J_X \cong \{ [q,x, q^\perp] \in B^{**} : x \in X \}$ isometrically, where
$$J_X = \{ x \in X : [q, x , q^\perp] = 0 \} = \{ x \in X : x = q^\perp x q^\perp \} .$$     Since $\| [q, x , q^\perp] \| = \inf \{ \| x + z \|
: z \in J_X \} \leq 1$, there exists $z \in J_X$ with $\| x + z \| \le 1 + \epsilon$.  
Set $y = x+z$.    \end{proof} 

The last result is saying that any element of $X$ whose off-diagonal corners in its $2 \times 2$ matrix 
form with respect to the projection $q$, have norm $\leq 1$, has the same off-diagonal corners as another
element of $X$ whose norm is close to $1$.

\begin{theorem} \label{supp}  Let $A$ be a closed  Jordan subalgebra of a $C^{\ast}$-algebra $B$. Suppose that a projection $p$ in $A^{\perp \perp}$ is open in $B^{\ast\ast}$. There exists a net $a_{t} \in {\rm Ball}(A)$ (even in $\frac{1}{2} {\mathfrak F}_A$), such that $p a_{t} p = a_{t}$ and $a_{t} \to p$ in the weak* topology.
Thus $D = \{ a \in A : pap = a \}$ is a hereditary subalgebra of $A$ with support projection
$p$, and $p$ is $A$-open in the sense of {\rm \cite{BWj}}.   
\end{theorem}

\begin{proof}  We let $q = p^\perp$, a closed projection in $B^{**}$.
First assume that $A$ is a  Jordan unital-subalgebra of a unital  $C^{\ast}$-algebra $B$.  
 From Corollary \ref{weakHay} we have a net $a^t_{n} \in A$, such that such that $qa^t_{n}q = 0$, $a^t_{n} \to q^\perp$ in the weak* topology, and $$\|qa^t_{n} q^\perp + q^\perp a^t_{n}q \| \le \frac{1}{ \sqrt{n}},
\qquad n \in \Ndb.$$  Write $[q,A^{\ast\ast}  , q^\perp]$ for $\{ [q, \eta,  q^\perp] : \eta \in A^{\ast\ast} \}$.  
Claim 1: if we set
$$Z = A \cap ([q,A^{\ast\ast}, q^\perp ] + q^\perp A^{\ast\ast} q^\perp)
= \{ a \in A : qaq = 0 \},$$ then $$Z^{\perp\perp} = [q,A^{\ast\ast}, q^\perp ] 
+ q^\perp A^{\ast\ast} q^\perp = \{ \eta \in A^{**} : q \eta q = 0 \}.$$  
If Claim 1 were true, then if $\psi \in Z^\perp$ in $(L+L^*)^*$ then 
$\psi \in (Z^{\perp \perp})_\perp$.   In particular $\psi$ annihilates 
$[q,A^{\ast\ast}, q^\perp ]$, and also the space $I_Z$ in  Theorem \ref{gamlemma2}.
Thus  the hypotheses of Theorem 
\ref{gamlemma2} hold with $X$ there replaced by $Z$.  
In the notation of that theorem, $J_Z = \{ x \in Z : [q, x , q^\perp] = 0 \}
= \{ a \in A : a = p a p \}$, which we will see is the HSA $D$ in $A$ with support projection $p$.  
By that theorem,  
$$\{ a \in A : qaq = 0 \}/D \cong \{ [p , a , q ]  \in B^{**} : a \in A,  qaq = 0 \} $$
completely isometrically, via the map $a + D \mapsto  [p , a , q ]  \in B^{**}$.
By Corollary \ref{coh}, since 
$$\| [q, a^t_n , q^\perp] \| = \inf \{ \| a^t_n + x \|
: x \in J_Z \} \leq  \frac{1}{ \sqrt{n}} ,$$ 
 there exists $j^t_{n} \in A$ with $q^\perp j^t_{n} q^\perp = j^t_{n}$ such that $\|a^t_{n} - j^t_{n} \| \le \frac{1}{n}$.  Hence $j^t_{n} \rightarrow 1-q = p$ in the weak* topology as desired.   This net may not be contractive, however
note that $D = \{ a \in A : p a p = a \}$ is a hereditary subalgebra of $A$ by the definition
at the start of 
\cite[Section 
3]{BWj}.  Hence by facts at the start of 
\cite[Section 
3]{BWj} it has a 
partial cai $(e_t)$, even  in $\frac{1}{2} {\mathfrak F}_A$, and $(e_t)$ will have weak* limit $p$.

We now prove Claim 1.   By Lemma \ref{tfqin}, we have
$A^\perp \subset (qAq)_\perp$.   It follows from Proposition \ref{gamlemma1} that $A/Z \cong qAq$ completely
isometrically.  Thus we obtain 
a complete isometry $A^{**}/Z^{\perp \perp} \to (qAq)^{**}$.   Note that 
$q +Z^{\perp \perp}$ maps to $q$ under this isometry.   One way to see 
this is to note that by the first lines of the present proof, $q^\perp \in Z^{\perp \perp}$.
Clearly   $q + Z^{\perp \perp} = q + 1-q  +Z^{\perp \perp} = 1 +Z^{\perp \perp}$, and this element 
maps to $q$ under the isometry above.

Clearly $Z^{\perp \perp} \subset \{ \eta \in A^{**} : q \eta q = 0 \}$.     Since $A$ is unital, $qAq$ is a unital
operator space (a unital-subspace of 
the unital
operator space $qA^{**}q$), and hence so is $(qAq)^{**}$.    Thus $A^{**}/Z^{\perp \perp}$
is a unital
operator space.
Suppose that $\eta \in 
(1-q)A^{\ast\ast}(1-q)$ but $\eta \notin Z^{\perp \perp}$.    Suppose that $\| \eta \| \leq 1$.
Then $$\| q + i^n \eta + Z^{\perp \perp} \| \leq \| q + i^n \eta \| = 1 < \sqrt{1 + \| \eta + Z^{\perp \perp} \|},
\qquad n \in \Ndb.$$
This contradicts the characterization of unital operator spaces in \cite{BNmetric}, since we 
know $A^{**}/Z^{\perp \perp}$ is a unital operator space.
Thus $(1-q)A^{\ast\ast}(1-q) \subset  Z^{\perp \perp}$.  

Now suppose that $\eta = [q , y , q^\perp ]$ for $y \in A^{**}$ but $\eta \notin Z^{\perp \perp}$.    Suppose that $\| \eta \| \leq 1/2$.
Then $$\| q + i^n \eta + Z^{\perp \perp} \| \leq \inf \{ \| q + i^n \eta + z_1 \| : 
z_1 \in Z^{\perp \perp} \cap A_1^{**}(q) \},$$
where $A_1^{**}(q) = \{ [q , y , q^\perp ] : y \in A^{**} \}$.   Claim 2: for any projection
$q$ and any operator $z$ on a Hilbert space, we have 
$$\| q + [q,z,q^\perp] \|   \leq \sqrt{1 +  \|  [q,z,q^\perp]  \|^2} .$$
Indeed $$(q + [q,z,q^\perp] ) (q + [q,z,q^\perp] )^* = q +[q,z,q^\perp] [q,z,q^\perp]^*
\leq 1 + \|  [q,z,q^\perp]  \|^2 .$$
It follows that 
$$\| q + i^n \eta + Z^{\perp \perp} \| \leq \inf \{ \sqrt{1 +  \|  \eta + z_1 \|^2} : z_1  \in Z^{\perp \perp} \cap A_1^{**}(q) \} .$$
Since we may assume that $ \|  \eta + z_1 \| \leq 1$, the last inequality is also true with $\|  \eta + z_1 \|^2$ 
replaced by $\|  \eta + z_1 \|$.   We make 
Claim 3: 
$$\inf \{  \|  \eta + z_1 \|  : z_1  \in Z^{\perp \perp} \cap A_1^{**}(q) \} 
=  \inf \{  \|  \eta +  z \| : z \in  Z^{\perp \perp} \} .$$
Indeed $\leq$ is obvious here.   On the other hand,  write $z  \in  Z^{\perp \perp} \subset \{ \eta \in A^{**} : q \eta q = 0 \}$ as 
$z = r + s$ with $r = [q,z,q^\perp]$ and $s = q^\perp z  q^\perp.$   Since 
$q^\perp A^{**}  q^\perp \subset Z^{\perp \perp}$, we have $s  \in  Z^{\perp \perp}$, and hence 
$r = z-s  \in  Z^{\perp \perp} \cap A_1^{**}(q)$.      Since $\eta \in A_1^{**}(q)$ it is
easy to see that $\|  \eta +  z \| \geq \|  \eta + r \| \geq \inf \{  \|  \eta + z_1 \|  : z_1  \in Z^{\perp \perp} \cap A_1^{**}(q) \}$.   Taking the infimum over such $z$ yields Claim 3.   

We now have 
$$\| q + i^n \eta + Z^{\perp \perp} \| \leq \sqrt{1 + \inf  \{  \|  \eta +  z \| : z \in  Z^{\perp \perp} \} }
= \| \eta + Z^{\perp \perp} \| , 
\qquad n \in \Ndb.$$
 Again this contradicts the characterization of unital operator spaces in \cite{BNmetric}, similarly to the above.
   To rule out the equality 
case of that characterization note that 
 in a unital operator space $X$ we have Claim 4: 
$\max_{n \in \Ndb} \, \| 1 + i^n x \| = \sqrt{1 + \| x \|}$ if and  only if $x = 0$.
Indeed tracing through the proof of that characterization  from \cite{BNmetric} 
yields that $$\sqrt{1 + \| x \|} = 
\max \{ \Vert \left[ \begin{array}{ccl} 2 & 2 x \\
0 & 2 \end{array}
 \right] \Vert , \Vert \left[ \begin{array}{ccl} 0 & 0 \\ 2x & 0 \end{array}
 \right] \Vert \}.$$
From equation
(2.3) in that paper, we deduce that
$$1 + \| x \|  
= \frac{1}{2}(2 + \Vert x \Vert^2  + \Vert x \Vert
 \sqrt{\Vert x \Vert^2 + 4}).$$
Solving this gives $\| x \| = 0$ and Claim 4.

Thus $\eta \in  Z^{\perp \perp}$.    Hence for any $y \in  A^{**}$ with $q \eta q = 0$ we
 have $y = [q , y , q^\perp ] + q^\perp y q^\perp \in Z^{\perp \perp}$.   Thus
$Z^{\perp \perp} = \{ \eta \in A^{**} : q \eta q = 0 \}$, proving Claim 1.

 If $A$ is unital but $1_A$ is not an identity for $B$ then replace $B$ by $C = C^*(A)$.
Then $p$ being open in $B^{**}$ implies that $p$ is open in $C^{**}$,  so that $p$ is $A$-open in 
$A$.  If $A$ is not unital
then consider $A^1$ as a Jordan unital-subalgebra of $B^1$ where $B^1 =B$ if $B$ is unital.
Then $p$ being open in $B^{**}$ implies that $p$ is open in $(B^1)^{**}$, so that $p$ is $A^1$-open in 
$A^1$.    That is, $p (A^1)^{**} p \cap A^1 = p A^{**} p \cap A^1 = p A^{**} p \cap A$ is
a hereditary subalgebra in $A^1$ and hence also in $A$.    (Here we are using the
fact that $A^{\perp \perp} \cap A^1 = A$.)  So $p$ is $A$-open.  \end{proof}

\begin{remark}   (1) \ At the end of the proof we used `open-ness' with respect to $A^1$ or with respect to 
a containing $C^*$-algebra $B$ (where $A^1 \subset B^1$).   This 
raises a concern since $A^1$ is not uniquely defined up to
complete isometry \cite{BWj2}.    The same concern arises at several later points in the 
present paper where $A^1$ is used.   However note that for a projection $p \in A^{**}$ the proof above 
will yield that if (the image  $\pi^{**}(p)$ of) $p$ is open in $B^{**}$  for an isometric Jordan homomorphism 
$\pi : A \to B$ for a $C^*$-algebra $B$, then $\pi^{**}(p)$ is $\pi(A)$-open in the sense of {\rm \cite{BWj}},
and hence $p$ is $A$-open.    The converse is much easier, if $p$ is $A$-open then $\pi^{**}(p)$  is open in $B^{**}$.
We may also replace $B$ by $B^1$ in the last lines.   Indeed  $p$ is $A$-open  iff $p$ is $A^1$-open for any
(or for every) unitization $A^1$.    Thus although matrix norms (i.e.\ operator space structure)
were used in the last proof, they are not usually relevant to 
statements of theorems in  the
noncommutative topology of Jordan operator algebras.

\smallskip 

(2)  \ We do not know when $q^\perp A^{**} q^\perp + q A^{**} q$ is the weak* closure of its 
intersection with the canonical copy of 
$A$ in $A^{**}$, for a (non-central) projection $q$ in $A^{**}$ (or if there are any interesting cases 
when this happens).  
\end{remark}

\begin{corollary} \label{sup}  Let $A$ be a closed Jordan subalgebra of a $C^{*}$-algebra $B$.  A projection $p$ in $A^{\perp \perp}$ is open in Akemann's sense in $B^{**}$ if and only if $p$ is $A$-open in the sense of {\rm \cite{BWj}}.   
\end{corollary}

\begin{theorem} \label{bhn29}  Let $A$ be a closed Jordan subalgebra of a $C^{\ast}$-algebra $B$.  There is a bijective correspondence between HSA's in $A$  and HSA's in $B$ with support projection in $A^{\perp \perp}$.
This correspondence  takes a  HSA $D$  in $A$ to the HSA $\overline{DBD^*}$ in $B$.   The inverse bijection is simply 
intersecting with $A$.     
\end{theorem}

\begin{proof}  If $D'$ is a  HSA in $B$ with support projection $p$ in $A^{\perp \perp}$   then 
by Theorem \ref{supp} and its proof there is a partial cai $(e_t)$ for a HSA $D$ in $A$ with the same support projection
$p$.
Also
$$D' \cap A = p B^{**} p \cap B \cap A =  p A^{**} p \cap  A = D.$$

Conversely, if $D$ is a HSA in $A$  with partial cai $e_t \to p$ weak*, then $DBD^*$ and $(DBD^*)^{\perp \perp}$ are contained in $pB^{**} p$.   Conversely, since $e_s b e_t^* \in DBD^*$  for $b \in B$, in the double weak* limit we see that
$p B p \subset (DBD^*)^{\perp \perp}$.  hence also $pB^{**} p \subset (DBD^*)^{\perp \perp}$, so that
$pB^{**} p = (DBD^*)^{\perp \perp}$.
Thus $\overline{DBD^*}$ has the same support projection $p$ as $D$.  
\end{proof}

\begin{corollary} \label{sepqr}   Let $A$ be a closed approximately unital Jordan subalgebra of a $C^{\ast}$-algebra $B$. Suppose that a projection $p$ in $A^{\perp \perp}$ is open in $B^{\ast\ast}$, and
let  $D = \{ a \in A : pap = a \}$ be a hereditary subalgebra of $A$ with support projection
$p$.   If $q = p^\perp$ then
$$\{ a \in A : qaq = 0 \}/D \cong \{ [p , a , q ]  \in B^{**} : a \in A,  qaq = 0 \}$$
completely isometrically, via the map $a + D \mapsto  [p , a , q ]  \in B^{**}$.
 \end{corollary}

\begin{proof}      We proved this   in  
the several lines after the statement of Claim 1 in Theorem \ref{supp} 
if  $A$ is a  Jordan unital-subalgebra of a unital  $C^{\ast}$-algebra $B$.
If the latter is not the case  then we can reduce to this case by the argument in the last paragraph of the proof
of  Theorem \ref{supp} (exercise).   Indeed note that the 
quantities $(1-p)a(1-p), D,$ and $[p , a , 1-p ]$ are unaffected by replacing $A$ or $B$ with unitizations.
 \end{proof}  

\section{Some distance  limit formulae}    
\begin{lemma} \label{decrin} In a $C^*$-algebra, if $\| b \| \leq 1$ and $\| ab  \| \leq \eps$
and $1 \geq a \geq c \geq 0$ then $\| c b \| \leq \sqrt{\eps}$.  A similar result holds 
with $ab$ replaced by $ba$.  \end{lemma}
\begin{proof}  This is because $\| b^* c^2 b \|  \leq 
\| b^* c b \| \leq \| b^* a b \| \leq \| ab  \| \leq \eps.$
 \end{proof}  

\begin{lemma}\label{incai}  Let $p$ be a open projection in $B^{\ast\ast}$ for a  $C^{\ast}$-algebra $B$.   
\begin{itemize} \item  [(1)] Suppose that $(f_t)_{t \in \Lambda}$ 
is an increasing  net in $B_+$ with weak* limit $p$.
 Then $(f_t)$ is a cai for the hereditary subalgebra of $B$ with support projection $p$.
\item  [(2)]  Suppose that $(e_t)_{t \in \Lambda}$ 
is a net in $\frac{1}{2}{\mathfrak F}_B$ with ${\rm Re} \, e_s \leq {\rm Re} \, e_t$ if $s \leq t$, 
and $e_t \to p$ weak*.   Then $(e_t)$ is a cai for the hereditary subalgebra of $B$ with support projection $p$.
\end{itemize} 
 \end{lemma}

\begin{proof}    Item  (1) is no doubt well known, but since we do not recall where it may be found we give a short proof.
Let $D$ be  the hereditary subalgebra of $B$ with support projection $p$.  
Given $F = \{ a_1, \cdots , a_n \} \subset D$ and $\eps > 0$, since $a^* (1-f_t) a \to 0$ weakly for each $a \in F$, by Mazur's theorem there exists a
convex combination $f = \sum_{k=1}^m \, \alpha_k \, f_{t_k}$ with $\| a^* (1-f) a \| \leq \eps$.  Choose $t_F^{\eps}$
dominating $t_1, \cdots, t_n$, then if $t \geq t_F^{\eps}$ we have $f_t \geq  \sum_{k=1}^m \, \alpha_k f_{t_k} = f$,
and $$\| a^* (1-f_t)^2 a \| \leq \| a^* (1-f_t) a \| \leq \| a^* (1-f) a \| \leq \eps .$$
It follows that $(f_t)$ is a left cai for $D$, and it is also a right cai by taking adjoints.

For (2),  let $f_t = {\rm Re} \, e_t$, then $f_t \nearrow p$ weak* too, so that  $(f_t)$ is a cai for $D$ by (1). 
That $e_t \in \frac{1}{2}{\mathfrak F}_B$ implies that
$e_t^* e_t \leq  \frac{1}{2}(e_t + e_t^*) = f_t$.  Then $$a^* (1-e_t)^* (1-e_t) a =
a^* (1-e_t^* - e_t + e_t^* e_t) a \leq a^*  (1 - f_t) a \to 0 .$$
It follows that $(e_t)$ is a left cai for $D$, and a symmetric argument 
shows that it is also a right cai. 
\end{proof}  

The last lemma is closely related the next results, but we gave a proof of it separately since it seems of independent interest.

\begin{theorem} \label{netprox}
Let $p$ be an open projection in $B^{\ast\ast}$ for a  $C^{\ast}$-algebra $B$.    Suppose that $(f_t)_{t \in \Lambda}$ 
is an increasing  cai for the hereditary subalgebra $D$ of $B$ with support projection $p$.   Let $q = p^\perp$. 
 If $x \in B$ with $qxq = 0$ then $(1-f_t) x (1-f_t) \to 0$.    
 \end{theorem}
\begin{proof}     
   We have $f_t \to p$ weak*.    Represent
$B$ nondegenerately on a Hilbert space $H$ in such a way that $B^{**}$ is a von Neumann 
algebra on $H$ and equal to the second commutant of $B$ (eg.\ the universal representation).   Then
$p$ becomes a projection on $H$, and is the WOT limit of $(e_t)$.    Since $D$ is
represented nondegenerately on $pH$ it follows that $f_t \to p$ SOT on $pH$ and hence also on $H$. 
Fix $x \in {\rm Ball}(B)$ with
$qxq = 0$, we have $(1-f_t) x (1-f_t) \to 0$ strongly, hence weak*.
Thus by Mazur's theorem as in the proof of Lemma \ref{incai},  given $\eps > 0$
there is convex combination $\sum_{k=1}^m \, \alpha_k \, (1-f_{t_k})x (1-f_{t_k})$ of norm $\leq \epsilon$.
Multiplying on the left by $q$ we get $\| q x (1-f) \| \leq \epsilon$ where
$f = \sum_{k=1}^m \, \alpha_k f_{t_k}$.   Choose $t_{\eps}$
dominating $t_1, \cdots, t_n$, then if $t \geq t_{\eps}$ we have $f_t \geq  \sum_{k=1}^m \, \alpha_k f_{t_k} = f$.

By  Lemma \ref{decrin}, since $\| q x (1-f) \| \leq \epsilon$, we deduce that $\| q x (1- f_t) \| \leq \sqrt{\eps}$ for $t \geq t_{\eps}$.
Similarly  there exists $t'_{\eps}$, with
$\| (1-f_t) x q \| \leq \sqrt{\eps}$ for $t \geq t_{\eps}$.  We may assume $t_{\eps} = t'_{\eps}$.

Write $x = x_0 + x_1$ where $x_0 = pxp$ and $x_1 = [p,x,q]$.   Then 
$$(1-f_t) x (1-f_t) = (1-f_t) x_0 (1-f_t) + (1-f_t) x_1 (1-f_t),$$ and 
$(1-f_t) x_0 (1-f_t) \in pB^{**}p$ and $(1-f_t) x_1 (1-f_t) \in B^{**}_1(q)$.  
We have
$$\| (1-f_t) x_1 (1-f_t) \| = \| (1-f_t) (q x + xq) (1-f_t) \| 
\leq 2 \sqrt{\eps} , \qquad t\geq t_{\eps} .$$  
Setting $z = (1-f_t) x (1-f_t) \in B$ we have $qzq = qxq = 0$, so that $z \in L+L^*$ in the notation 
seen in the second paragraph of the proof of Proposition \ref{gamlemma1}.
  Also $$[p,z,q] = p(1-f_t) x_1 (1-f_t)q + q (1-f_t) x_1 (1-f_t) p = (1-f_t) (px_1q + q x_1 p)(1-f_t)$$
which is simply $(1-f_t) x_1 (1-f_t) .$
Thus $\| [p,z,q]  \| \leq 2 \sqrt{\eps}$.   By Proposition \ref{cl01d}   there exists 
$d \in D$ with $\| z + d \| \leq 3 \sqrt{\eps}$.  Pre- and post-multiplying by $p$ we obtain
$\| (1-f_t) x_0 (1-f_t) + d \| \leq 3 \sqrt{\eps}$, for all $t \geq t_{\eps}$.   By Lemma
\ref{incai} there exists $t'_{\eps}$, with
$\| d (1-f_s)  \| \leq \sqrt{\eps}$ for $s \geq t_{\eps}$.  We may assume $t_{\eps} = t'_{\eps}$.  Thus for $s, t \geq t_{\eps}$ we have
$$\| (1-f_t) x_0 (1-f_t) (1-f_s) \|
\leq  \| (1-f_t) x_0 (1-f_t) (1-f_s) + d  (1-f_s) \| + \sqrt{\eps}  \leq 4 \sqrt{\eps}.$$ 
By Lemma \ref{decrin}  we obtain 
$$\| (1-f_s) x_0 (1-f_t) (1-f_s) \|
\leq  2 \epsilon^{\frac{1}{4}} , \qquad s \geq t \geq t_{\eps}.$$ 
Setting $t = t_{\eps}$, there exists $\lambda_{\eps} \geq t_{\eps}$ with
$\| f_{t_{\eps}} (1-f_s)  \| \leq \eps^{\frac{1}{4}}$ for $s \geq \lambda_{\eps}$.  
Thus
$$\|  (1-f_s) x_0  (1-f_s) \| \leq \| (1-f_s) x_0 (1-f_{t_{\eps}}) (1-f_s) \| +  \| (1-f_s) x_0 f_{t_{\eps}} (1-f_s) \|
\leq 3 \eps^{\frac{1}{4}} ,$$ for $s \geq \lambda_{\eps}.$ 
We have already seen that 
$$\| (1-f_t) x_1 (1-f_t) \| = \| (1-f_t) (q x + xq) (1-f_t) \| 
\leq 2 \sqrt{\eps} , \qquad t\geq t_{\eps} .$$  
Thus we have 
$$\| (1-f_t) x (1-f_t) \| = \| (1-f_t) (q x + xq) (1-f_t) \| 
\leq 3 \eps^{\frac{1}{4}} + 2 \sqrt{\eps} , \qquad t\geq \lambda_{\eps} .$$  
This means that $(1-f_t) x (1-f_t) \to 0$ in norm.  
\end{proof}

\begin{corollary} \label{coofwasl}  
Let $q$ be a closed projection in $B^{\ast\ast}$ for a  $C^{\ast}$-algebra $B$.    Suppose that $(e_t)_{t \in \Lambda}$ 
is a net in $\frac{1}{2}{\mathfrak F}_B$ with ${\rm Re} \, e_s \leq {\rm Re} \, e_t$ if $s \leq t$, 
and $e_t \to q^\perp$ weak*.   If $x \in B$ with
$qxq = 0$ then $(1-e_t) x (1-e_t) \to 0$.
 \end{corollary}
\begin{proof}    
  Let $f_t = {\rm Re} \, e_t$, then $f_t \nearrow p= q^\perp$ weak*.   By the previous result,
$(1-f_t) x (1-f_t) \to 0$ in norm.   Hence as in the last lines of Lemma \ref{incai},
 $$(1-f_t)^* x^* (1-e_t)^*(1-e_t) x (1-f_t) 
\leq (1-f_t)^* x^* (1-f_t) x (1-f_t) \to 0 .$$
Thus $(1-e_t) x (1-f_t)  \to 0$.   Hence, and by a similar argument,
$$(1-e_t) x (1-e_t)(1-e_t)^* x^* (1-e_t)
\leq (1-e_t) x (1-f_t) x^* (1-e_t) \to 0 .$$
Thus $(1-e_t) x (1-e_t) \to 0$.  
 \end{proof}

In connection with the next result we note that it is known that if $J$ is a closed right ideal in a $C^*$-algebra $B$
and $a \in B$ then the distance from $a$ to $J$ equals $\lim_t \| (1-f_t) a \|$ where $(f_t)$ is a left approximate
identity 
for $J$ (or equivalently, an approximate
identity 
for the  hereditary subalgebra  $J \cap J^*$).   See e.g.\ \cite[Lemma 3.12]{Lin}.   We give some variants of this fact:
 
\begin{corollary} \label{qnf} 
Let $p$ be an open projection in $B^{\ast\ast}$ for a  $C^{\ast}$-algebra $B$,
and let $q = p^\perp$ in $B^{**}$.  
\begin{itemize} \item  [(1)] Suppose that $(f_t)_{t \in \Lambda}$
is an increasing cai for the hereditary subalgebra of $B$ with support projection $p$.
If $I = \{ b \in B : qbq = 0 \}$ and if $x \in B$ then
the distance from $x$ to $I$
equals $\lim_t \, \| (1-f_t) x (1-f_t) \| = \| q x q \|$.
\item  [(2)]  Suppose that $A$ is a Jordan subalgebra of $B$ with $q \in A^{\perp \perp}$.
Suppose that $(e_t)$ is a partial 
cai for $A$ in $\frac{1}{2}{\mathfrak F}_A$ with ${\rm Re} \, e_s \leq {\rm Re} \, e_t$ if $s \leq t$ 
(this can always be arranged by {\rm \cite[Proposition 4.6]{BWj}}).  
Or more generally suppose that $(e_t)$ is a
net in $\frac{1}{2}{\mathfrak F}_A$ with weak* limit $p$,
and with $({\rm Re} \, e_t)$   increasing in $B$.
If $I = \{ a \in A : qaq = 0 \}$ and if $x \in A$ then
the distance from $x$ to $I$
equals $\lim_t \, \| (1-e_t) x (1-e_t) \| = \| q x q \|$.
\end{itemize}
\end{corollary} 

\begin{proof}   We just prove (2) since (1) follows from (2).
 By Proposition \ref{gamlemma1} with $X = A$ (and using Lemma \ref{tfqin}),
 if $I = \{ x \in A : qxq = 0 \}$ and $a \in A$, the distance $d(a, I)$ from $a$ to $I$ equals $\| q aq \|$.
Given $\eps > 0$ there exists $d \in I$ with $$\| qa q \| > \| a + d \| - \epsilon \geq
\| (1-e_t) (a+d) (1-e_t) \|   - \epsilon .$$
By Corollary \ref{coofwasl}
we have $(1-e_t) d (1-e_t) \to 0$.   Hence there exists $t_0$ such that 
$$\| qa q \| > \| (1-e_t) a (1-e_t) \|   - 2 \epsilon  \geq \| q  (1-e_t) a (1-e_t) q \|   - 2 \epsilon
= \| q a q \|   - 2 \epsilon,$$ for $t \geq t_0 .$   Thus $\| (1-e_t) a (1-e_t) \| 
\to \| q a q \|$.
 \end{proof}

 There is a variant of Corollary \ref{qnf} with $I$ replaced by a HSA
$D$.    

\begin{corollary} \label{qnf2} 
Let $p$ be an open projection in $B^{\ast\ast}$ for a  $C^{\ast}$-algebra $B$,
and let $q = p^\perp$ in $B^{**}$.  
\begin{itemize} \item  [(1)]   Suppose that $a \in B$ with $[q , a , q^\perp ] = 0$.    By the $C^*$-algebra 
case of Lemma  \ref{distH}  the distance $d(a, D)$ from $a$ to $D$ equals $\| qx q \|$.
By Corollary \ref{qnf} this equals $\lim_t \, \| (1-f_t) x (1-f_t) \|$ for 
any increasing cai for $D$. 
\end{itemize}
\end{corollary} 

  Similarly, if $A$ is a Jordan subalgebra of $B$ with $q \in A^{\perp \perp}$, and if $D$ is the HSA in $A$ 
with support projection $p = q^\perp$,
then for any net $(e_t)$ as in (2) of the corollary,
and for any $a \in A$  with $[q , a , q^\perp ] = 0$,
the distance $d(a, D)$ from $a$ to $D$ equals $\| qx q \|$ by Lemma  \ref{distH},
which equals $\lim_t \, \| (1-e_t) a (1-e_t) \|$
by  Corollary \ref{qnf}.

\section{Initial consequences} \label{consq}

 It follows from some of 
the results in Section \ref{Jcompl} that essentially all, with (possibly) a very few 
exceptions, of the results on noncommutative topology,
noncommutative peak sets,  noncommutative peak interpolation, etc,  from the papers \cite{BHN, BRI, BRII, BRord, BNII,Bnpi, Hay} and 
others, generalize very literally to  Jordan operator algebras.    We describe these results in more or less
chronological order, and this will comprise almost all of the remainder of our paper.   

 We have already discussed in Section 4 the Jordan  generalizations of Hay's main results from \cite{Hay}.  
In particular we recollect Corollary \ref{sup} from Section 4:
 if $A$ is a closed Jordan subalgebra of a $C^{*}$-algebra $B$, then a projection $p$ in $A^{\perp \perp}$ is open in Akemann's sense in $B^{**}$ if and only if $p$ is $A$-open in the sense of {\rm \cite{BWj}} (that is, there exists a net in $A$ with
$x_t = p x_t p \to p$ weak* in $A^{**}$). 
If $A$ is an approximately unital  Jordan operator algebra then we recall that a projection $q$ in $A^{**}$ is {\em closed} if 
$1-q$ is open in $A^{**}$, where $1$ is the identity of $A^{**}$.   This is equivalent to $q$ being closed in Akemann's sense in $B^{**}$.

  The generalization of results from \cite{Hay} on peak projections will be done in Section \ref{Peak}.    

  We next 
discuss the  Jordan variants of results from paper  \cite{BHN} that were not covered in \cite{BWj}.  We
note that the just stated Corollary \ref{sup} is the Jordan variant
of  \cite[Theorem 2.4]{BHN}: thus what we called $A$-open projections in \cite{BWj}
are just the open projections with respect to (any) containing $C^*$-algebra $B$
which lie in $A^{\perp \perp}$.    Therefore,
 as in the associative operator algebra case we will drop the $A$- prefix,
and simply call them open projections with respect to $A$, or open projections in $A^{**}$, 
or simply open projections if there is little chance of confusion.   Hence all
the $A$- prefixes to `open' in results in \cite{BWj} may be dropped, and the proofs in \cite{BWj} of such 
results can often be simplified.   Thus for example part of Lemma 3.12 in \cite{BWj} becomes:
the supremum in $A^{**}$ (or in $B^{**}$) of any collection of  open projections in $A^{**}$
is  open in $A^{**}$.   And the now simple proof of the latter is simply that the supremum in $B^{**}$
is open in $B^{**}$ by Akemann's $C^*$-theory, and it is also in $A^{\perp \perp}$ (since as we said
towards the end of Section 1 in \cite{BWj}, $A^{\perp \perp}$ is closed under meets and joins of projections).   
Theorem 2.9 in \cite{BHN}
was generalized in Theorem \ref{bhn29}.     Similarly \cite[Theorem 3.15]{BWj}, together with the fact above that open and $A$-open projections coincide, becomes 
the generalization to an approximately unital  Jordan operator algebra $A$ of  \cite[Theorem 4.1]{BHN}, namely 
it gives the link between open projections, and weak* closed faces and lowersemicontinuity in the 
 quasistate space $Q(A)$.  If $A$ is unital then there is a similar result and proof using the state space $S(A)$.
Indeed we also have:

\begin{lemma} \label{mds}  Let $A$ be an approximately unital Jordan operator algebra,
and let $q$ be a projection in $A^{\perp \perp}$.
Then
$$\{ \varphi \in S(A) : \varphi(q) = 1 \} = qA^* q \cap S(A) ,$$
and this is a face in the state space of $A$.     If we write this face as $F_q$, then 
we have $q_1 \leq q_2$ if and only if $F_{q_1} \subset F_{q_2}$, for projections $q_1, q_2 \in A^{**}$.
\end{lemma}

\begin{proof} Let $A$ be a closed approximately unital Jordan subalgebra of a $C^*$-algebra $B$, and suppose that $B$ is generated by $A$.    Suppose that
$\varphi \in S(A)$ and if $\varphi(q) = 1$.   If $\hat{\varphi}$ is a state extension of $\varphi$ to  $B$, viewed as a state of $B^{**}$,
then $q$ is
in the multiplicative domain of $\hat{\varphi}$ and so $\varphi(qxq) = \varphi(x)$ for all $x$.   That is, $\varphi \in qA^* q \cap S(A)$.   Conversely, if $\varphi \in qA^* q \cap S(A)$ then $1 = \varphi(1) = \varphi(q)$.   Thus $qA^* q \cap S(A) = \{ \varphi \in S(A) : \varphi(q) = 1 \}$. 

The final `iff'  is stated  in \cite[Proposition 3.16]{BWj} for open projections with a very sketchy proof.  For convenience we give details.  If $q_1 \leq q_2$ are projections in $A^{**}$ and $\varphi \in F_{q_1}$ then $1 = \hat{\varphi}(q_1)
=   \hat{\varphi}(q_1 q_2 q_1) = \hat{\varphi}(q_2)$ by the last paragraph.   So $F_{q_1} \subset F_{q_2}$.
Conversely, suppose that  $F_{q_1} \subset F_{q_2}$  but  $q_1$ is not dominated
by $q_2$.   Suppose  that $B$ is a nondegenerate $*$-subalgebra of $B(H)$ with 
$B^{**} \subset B(H)$ as a von Neumann algebra.   The range of $q_1$ in $H$ is not a
subset of the range of $q_2$.   Choose a unit vector $\xi$ in the range of $q_1$ 
but not in the range of $q_2$, and define $\varphi(x) = \langle x \xi , \xi \rangle$ for $x \in A$.   The unique weak* continuous extension 
of $\varphi$ to $A^{**}$ is of the same form. 
Since $A$ acts nondegenerately, if $(e_t)$ is a cai for $A$ then
$\langle e_t \xi , \xi \rangle \to 1$.  So $\varphi \in S(A)$.
We have $\varphi(q_1) = \langle \xi , \xi \rangle  =1$, and so $\varphi \in F_{q_1} \subset F_{q_2}$. 
Hence $\varphi(q_2) = \langle q_2 \xi , \xi \rangle  =1$,  
so  $q_2 \xi = \xi$. 
This is a contradiction.  So $q_1 \leq q_2$. 
 \end{proof}

The last statement of the last result is true with $S(A)$ replaced by $Q(A)$ in the definitions, with the same proof.

We will use the last result in our proof of another interesting characterization of 
closed (and hence open) projections in an approximately unital  Jordan operator algebra, 
one that does not reference approximation of the projection by elements in $A$.   This result may be new even
if $A$ is a $C^*$-algebra.   It is a variant of \cite[Remark 2.5 (ii)]{BHN}.

\begin{theorem} \label{qbsq}  Let $A$ be an approximately unital  Jordan operator algebra.   A projection $q$ in $A^{**}$
is 
closed  (that is, $q^\perp$ is open  in $A^{**}$)
 if  and only if $qA^* q$  is weak* closed in $A^*$.  
 \end{theorem}

\begin{proof}      
If  $q$ in $A^{**}$ then it is an exercise to show that $qA^*q = W_\perp$ where
$W = \{ \eta \in A^{**} : q \eta q = 0 \}$.     Let $p = q^\perp$.

Suppose that $qA^* q$  is weak* closed in $A^*$. 
Then  $qA^* q$ equals $E^{\perp}$ for some subspace $E$ of $A$.   
Hence $E^{\perp \perp} = (qA^*q )^\perp = (W_\perp)^\perp =  W$, and so $E = W \cap A = Z$ where $Z$ is 
the space in the last paragraph of  the proof of Theorem \ref{supp}.  

Let  $B$ be a $C^*$-algebra containing and generated by $A$ as a closed Jordan subalgebra.    To show
that  $q$ is closed in $A^{**}$
 it suffices by
our variant Corollary \ref{sup} of Hay's theorem to show that $q$ is closed in $B^{**}$.    It follows from the lemma and
facts at  the end of 
 \cite[Section 2]{APfaces} that $q$ is a closed projection in $B^{**}$ if  and only if the face 
$$F_q =  \{ \varphi \in S(B) : \varphi(q) = 1 \} = qB^* q \cap S(B)$$ is weak* closed  in $S(B)$.   
 So let $\varphi_t \in F_q$,
with $\varphi_t \to \varphi  \in S(B)$ weak* in $B^*$.  
Since  $\varphi_t \in Z^\perp$, we have $\varphi \in Z^\perp = W_\perp$.
Thus $\varphi(p) = 0$ since  $p \in W$. 
It follows that $\varphi \in F_q$, and so $F_q$  is weak* closed in $S(B)$ as desired.

Conversely, suppose that  $q$ 
is closed.  First suppose that $A$ is unital.  By Claim 1 in the proof of Theorem \ref{supp}, if $Z$ is the space defined in that Claim then
$Z^{\perp \perp} = W$.   Hence $qA^*q = W_\perp = Z^{\perp}$ by the bipolar theorem, so that
$qA^* q$  is weak* closed in $A^*$ as desired.  

Now assume that $A$ is  nonunital.      Let $r = 1-p \in (A^1)^{**}$,  then $r (A^1)^* r$ is weak* closed in
$(A^1)^*$ by the last paragraph.  We wish to show that $qA^* q$  is weak* closed in $A^*$. 
Suppose that $\varphi_t = q \varphi_t q \in qA^* q$ and that 
$\varphi_t \to \varphi \in A^*$ weak*.     By the Krein-Smulian theorem we may suppose that $(\varphi_t)$ is bounded.
 Define $\psi_t \in (A^1)^*$ by $\psi_t(x) = \varphi_t(qxq)$.   
Suppose by Alaoglu's theorem that a  subnet $(\psi_{t_\nu})$ converges weak* to $\psi \in (A^1)^*$.
 Then $(\psi_{t_\nu})$, and hence $\psi$, is in  $r (A^1)^* r$.   Indeed since $q \leq r$  we have $\psi_t(rxr) = \varphi_t(qrxrq) = \varphi_t(qxq)
= \psi_t(x)$.   Also, $$\psi_{t_\nu}(a + \lambda 1) = \varphi_{t_\nu}(qaq) + \lambda \varphi_{t_\nu}(q)
\to \varphi(a) + \lambda \psi(1), \qquad a \in A, \lambda \in \Cdb.$$
Thus $\varphi(a) = \psi(a)$, so $\varphi = \psi$ on $A^{**}$.
Hence $\varphi(a) = \psi(rar) = \psi(qaq) =   \varphi(qaq)$ as desired.
\end{proof} 

\begin{remark}   We remark that there is a typo in a related result in \cite{BNII}.   Namely in \cite[Proposition 4.3 (3)]{BNII} the last $S(A)$ should be $Q(A)$.   In the third last line of that proof $S(B)$ should be $Q(B)$.  \end{remark}

Section 6 from \cite{BHN} is largely concerned with peak projections and some 
of this will be generalized in Section 
\ref{Copeak} below.  We will not return in the present paper to explicitly  generalizing particular results from  
 \cite[Section 6]{BHN} since some of this is subsumed by later theory, and also because 
many things about peak projections are clear from the $C^*$-algebra
case, or follow easily from computations in the (associative) operator algebra generated
by the real positive element which is peaking, so need no Jordan variant.   This is often the
case in results involving the projections $s(\cdot)$ and $u(\cdot)$, as we said in the introduction.

We now turn to Jordan versions of results from \cite{BRI}.  The fact from  \cite{BRI} that open projections in $A^{**}$ are simply the suprema of support
projections of elements in ${\mathfrak F}_A$ (or equivalently of real positive  elements in $A$,
by e.g.\ \cite[Corollary 3.6]{BRII}), is generalized in \cite[Lemma 3.12 (1)]{BWj}, if we use the fact above that 
open and $A$-open projections coincide.

We recall that a projection $e$ is a Jordan operator algebra $A$ commutes with an element $x$ if
$e \circ x = exe$.   If  $A$ is a Jordan subalgebra of a $C^*$-algebra $B$ then by the first labelled equation in \cite{BWj} these relations
force $e$ to commute with $x$ in $B$: $ex = xe$.

\begin{proposition} \label{commop}
Let  $A$ be an approximately unital Jordan operator algebra.
The  product of two commuting open projections in $A^{**}$  is open, and this product is the support of the intersection $J \cap K$ of the
matching hereditary subalgebras $J, K$. 
\end{proposition}

\begin{proof}  Let $e, f$ be commuting open projections in $A^{**}$, with 
matching HSA's $J, K$.     If $A$ is a Jordan subalgebra of a $C^*$-algebra $B$ then 
$e, f,$ and $ef$ are open in $B^{**}$, and  $ef \in A^{\perp \perp} \subset B^{**}$ (see e.g.\
the last paragraphs of \cite[Section 1]{BWj}).    Thus $ef$ is open in $A^{**}$ by Corollary \ref{sup}.  
Note that $ef A^{**} ef \subset (e A^{**} e) \cap (f  A^{**}  f)$.  Conversely, if 
$x \in (e A^{**} e) \cap (f  A^{**}  f)$ then $ef x fe = x$.   Thus 
$ef A^{**} ef = (e A^{**} e) \cap (f  A^{**}  f)$, and so also
$$ef A^{**} ef \cap A = (e A^{**} e) \cap (f  A^{**}  f) \cap A = J \cap K.$$
Thus $J \cap K$ is a HSA in $A$ with support projection $ef$. 
\end{proof} 

The Urysohn lemma in \cite[Theorem 2.24]{BRI} will be discussed  in Section 
\ref{Copeak} below, and some generalizations of the results in Sections 3, 4, and 7 of \cite{BRI} will be given in 
\cite{BWj2,ZWdraft}. Concerning Section 6 from \cite{BRI}, the first result there is generalized in
\cite[Proposition 3.28]{BWj}.  We also have the analogues of part of 6.2 and 6.3 from \cite{BRI}.   But first we
prove an analogue of \cite[Proposition 3.1]{BRoyce}:

\begin{lemma} \label{broyce}  Let $J$ be an approximately unital  closed Jordan ideal in a Jordan operator algebra $A$.
Then $A/J$ is approximately unital if and only if $A$ is  approximately unital.  
\end{lemma}

\begin{proof} ($\Rightarrow$) \   If $A$ has a J-cai $(e_t)$ then clearly $(e_t + J)$ is a J-cai in $A/J$.   Thus the result 
follows from \cite[Lemma 2.5]{BWj}.   

($\Rightarrow$) \   Suppose that $A/J$ is approximately unital.   Let $p$ be the support projection of $J$ in
$A^{**}$, which is central.   Then $1-p$ is a contractive projection in $(A^1)^{**}$.
The canonical projection $A^{**} \to A^{**} (1-p)$ given by right multiplication by $p$ is a Jordan morphism
 with kernel $A^{**} p = J^{\perp \perp}$.   Thus  $A^{**} / A^{**} p \cong A^{**} (1-p)$.
Also $(A/J)^{**} \cong A^{**} / A^{**} p$ as Jordan algebras, so that the latter, and hence also $A^{**} (1-p)$, 
is unital.   Thus $A^{**} =  A^{**} p \oplus^\infty A^{**} (1-p)$ is unital.    Hence $A$ is approximately 
unital by \cite[Lemma 2.6]{BWj}.   \end{proof}  

\begin{proposition} \label{bri62}    Let $J$ be an approximately unital  closed Jordan ideal in a Jordan operator algebra $A$, and let $q : 
A \to A/J$ be the canonical quotient map.   
\begin{enumerate}
\item  [(1)]  Any closed approximately unital Jordan subalgebra $D$ in $A/J$ is the image under $q$ of 
a closed approximately unital Jordan subalgebra of $A$.  Indeed $q^{-1}(D)$ will serve here.
\item    [(2)]   The HSA's in $A/J$ are exactly the images under  $q$ of HSA's in
$A$.
\item    [(3)]   The open projections in $(A/J)^{**}$ are exactly the $q^{**}(p)$, for  open projections $p$ in
$A^{**}$.
\end{enumerate}  
\end{proposition}

\begin{proof}  (1) \ Note that $J$ is an  approximately unital  closed Jordan ideal in the closed Jordan subalgebra
$q^{-1}(D)$ of $A$.   Moreover $q^{-1}(D)/J \cong D$ isometrically, and $D$ is  approximately unital.   
Hence $q^{-1}(D)$ is approximately unital by Lemma \ref{broyce}. 

(2)  \ 
Since $q$ is a contractive surjective  Jordan morphism it is easy to see that   the image under  $q$ of a HSA in
$A$ is a HSA.     Conversely if $D$ is a  HSA in $A/J$ then $q^{-1}(D)$ is an approximately unital 
subalgebra of $A$ by (1).   Moreover if $q(a) \in D$ and $b \in A$ then $q(aba) = q(a) q(b) q(a) \in D$,
so that $aba \in q^{-1}(D)$.   Hence $q^{-1}(D)$ is a HSA in $A$.

(3) \ If $p$ is open in $A^{**}$ then $q^{**}(p)$ is open as in the proof of \cite[Proposition 6.3]{BRI}. 
Conversely if $r$ is an open projection in $(A/J)^{**}$  then $r$ supports a HSA $D$ 
in $(A/J)^{**}$.   By (2) there is a HSA $E$ in $A$ with $q(E) = D$.   If $p$ is the support projection
of $E$ then 
$q^{**}(pA^{**}p) = r (A/J)^{**} r$ and 
$q^{**}(p) = r$ since $p$ and $r$ are the identities of these bidual Jordan algebras 
and $q^{**}$ is a Jordan morphism.  \end{proof} 

This result will be used in Section \ref{sec9}.

\section{Compact projections}
 \label{Copeak}     
Throughout this section $A$ is a Jordan operator algebra, and  $B$ is a $C^*$-algebra containing $A$ as a closed Jordan subalgebra.   
If $A$ is approximately unital  then a  closed projection $q\in A^{**}$ will be called {\em compact} in $A^{**}$ if there exists $a\in {\rm Ball}(A)$ with $q=
q \circ a$.   It is easily seen by considering the $2 \times 2$ matrix of $a$ with respect to $q$ that the latter is equivalent to $q = qaq,$ and also to $q = aq$ (or $q = qa$) in $B$.    
We say such $q$ is {\em positively compact} in $A^{**}$ if the element $a$ above may be chosen in $\frac{1}{2}{\mathfrak F}_A.$    Note that by taking $n$-th roots of this $a$ we can ensure that it is as close as we like to 
the set ${\rm Ball}(B)_+$ (that is, {\em nearly positive} in the sense e.g.\ of \cite{BRord}).   Any compact projection $q$ in $A^{\ast\ast}$ is compact  in $B^{\ast\ast}.$ Clearly any closed projection is compact and positively compact  in $A^{\ast\ast}$ if $A$ is unital. 
Any closed projection dominated by a compact projection in $A^{\ast\ast}$ is compact.
 If $q$ is a compact projection with $q = qaq$ for $a\in {\rm Ball}(A)$, and if $(e_t)$ is a partial cai in $\frac{1}{2} {\mathfrak F}_A$
for the HSA supported by $e-q$ where $e=1_{A^{\ast\ast}},$ then $e-e_t\to q$ weak*. Let $y_t=a \circ (e-e_t)\in {\rm Ball}(A).$ Then $y_t\to a \circ (e-(e-q))=a \circ q=q$ weak*.  We have 
$q y_t q=\frac{1}{2}(q a(q-e_tq)+ q(1-e_t)aq)=q,$ and similarly $qy_t=q.$   

\begin{theorem} \label{1Ur} {\rm 
(A first noncommutative Urysohn lemma for approximately unital Jordan operator algebras)} \ Let $A$ be an approximately unital Jordan operator algebra, a closed Jordan subalgebra of $C^*$-algebra $B,$ and let $q$ be a compact projection in $A^{**}.$ Then for any open projection $u\in B^{**}$ with $u\leq q,$ and any $\varepsilon>0,$ there exist an $a\in {\rm Ball}(A)$ with $qaq=q$ and $\Vert a(1-u)\Vert <\varepsilon$ and $\Vert (1-u) a\Vert< \varepsilon.$ 
(The latter products are in $B$.)   \end{theorem}

\begin{proof}
Let $q\in A^{\perp\perp},$ let $u$ be an open projection with $q\leq u,$ and let $\varepsilon>0$ be given. By e.g.\ the lines above the theorem, there exists a net $(y_t)$ in ${\rm Ball}(A)$ with $q y_t q=q$ and $y_t\to q$ weak*. As in the proof of \cite[Theorem 2.1]{BNII},
given $\varepsilon>0$ there is a convex combination $a$ of the $y_t$ with
$qaq=q$ and $\Vert a(1-u)\Vert < \varepsilon$ and $\Vert (1-u)a\Vert <\varepsilon.$  
\end{proof}

This result and the following one were found by the first author and Zhenhua Wang.

\begin{theorem}  \label{chcomp}  Let $A$ be an approximately unital closed Jordan subalgebra of a $C^*$-algebra $B.$ If $q$ is a projection in $A^{**}$ then the following are equivalent:
\begin{enumerate}
\item [(i)] $q$  is compact in $B^{**}$ in the sense of Akemann,
\item [(ii)] $q$ is a closed projection in $(A^1)^{**},$
\item [(iii)] $q$ is compact in $A^{\ast\ast},$
\item [(iv)] $q$ is positively  compact in $A^{**}.$
\end{enumerate}
        \end{theorem}
\begin{proof}   We merely discuss the emendations that need to be made to the
proof of \cite[Theorem 2.2]{BNII}--we will omit mentioning
arguments which are identical.  
As in that proof we may assume that $A$ is not unital (if 
$A$ is unital then compact is the same as closed, and then this
follows from Corollary \ref{sup}), and that 
$1_{A^1} = 1_{B^1}$.  The proof that (iii) implies (ii)
relies as in \cite{BNII} on the lines above Theorem \ref{1Ur} concerning
$(y_t)$ together
with the fact that $1-y_t = (1-q) (1-y_t) (1-q)$,
which is clear since $q y_t = y_t q = q$ as we said after the definition of 
compact  projections above.  Thus taking the limit with $t$,
we see that $1-q \in ((1-q) (A^1)^{**} (1-q) \cap A^1)^{\perp \perp}$,
which says that $1-q$ is open and so $q$ is closed in $(A^1)^{**}$.  
 If $q$ is closed in $(A^1)^{**}$ 
then $q$ is closed in $A^{**}$  
since $e-q = e(1-q)$ is open by Proposition \ref{commop}.
Here $e = 1_{A^{**}}$.
   The hard direction is that 
(ii) implies (iv).  Following the argument for this
in \cite[Theorem 2.2]{BNII} we obtain that $f = 1-e$ is central
and $f A^1 f \in \Cdb f$.   Also $e(1-q)$ is open as we said above,
so supports a HSA $D$ in $A^1$ which by the argument we are following is
an approximately unital Jordan ideal (by e.g.\ \cite[Theorem 3.25]{BWj}) 
in the HSA $C$ in $A^1$ supported by $1-q$.  Since  multiplying
by the central projection $f$ on $C^{**}$ has kernel $D^{\perp \perp}$
we obtain the centered equation in the proof we are following.
At some point we have to appeal to \cite[Proposition 3.28]{BWj} in place of the
reference to \cite{BRI} there.  The last line of the
proof we are following should be changed to $q(1-b)q = q$, and this holds because
$qbq = 0$ since $b \in C$.  
        \end{proof}

It follows just as in \cite[Lemma 6.1]{BRII} that a closed projection $q$ for an approximately unital Jordan
operator algebra $A$
is compact in $A^{**}$  if and only if $q = qx$ (or $q=xq$)
for some $x \in A$.   The product here is in any $C^*$-algebra  containing $A$ as a closed Jordan subalgebra.

If $A$ is not approximately unital  then a  projection $q\in A^{**}$ will be called {\em compact} in $A^{**}$ if it is
closed in $(A^1)^{**}$ with respect to $A^1$.   This is equivalent,  if 
$A$ is a  Jordan subalgebra of a $C^*$-algebra $B$, to the projection $q\in A^{**}$ being
compact in $B^{**}$.  For by Theorem \ref{chcomp}, $q$  is compact in $A^{**}$  if and only if 
$q$ is closed in $(B^1)^{**}$, which by the $C^*$-theory happens 
 if and only if  $q$  is compact in $B^{**}$.

\begin{corollary} \label{bnc23} Let $A$ be a Jordan operator algebra.  The infimum of any family of compact projections in $A^{**}$ is
compact  in $A^{**}$.  The supremum of two commuting compact projections in $A^{**}$ is
compact  in $A^{**}$.
\end{corollary}

\begin{proof}  Let $A$ be a  Jordan subalgebra of a $C^*$-algebra $B$,
then this follows from the $C^*$-algebra case of the present result,
the fact above Corollary \ref{bnc23} that compactness of $q$ in  $A^{**}$
is equivalent to  being compact in $B^{**}$, and the 
fact towards the end 
of Section 1 in \cite{BWj} that $A^{\perp \perp}$ is closed under meets and joins of projections.  \end{proof}

\begin{corollary} \label{bnc24}  Let $D$ be a closed Jordan subalgebra of
a  Jordan
operator algebra $A$.
A projection $q \in D^{\perp \perp}$ is compact in $D^{**}$  if and only if $q$  is
compact  in $A^{**}$.    
\end{corollary}

\begin{proof}   Use the fact above Corollary \ref{bnc23} that compactness of $q$ in $D^{**}$ or  $A^{**}$
is equivalent,  if 
$A$ is a  Jordan subalgebra of a $C^*$-algebra $B$, to  being
compact in $B^{**}$. 
  \end{proof}
Corollary 2.5 of \cite{BNII} is clearly misstated (it is correct in the ArXiV version).  
It is generalized by  the following result, which is immediate from Corollary
\ref{bnc24}:
 
\begin{corollary} \label{bnc25}   Let $A$ be a Jordan operator algebra.
If a projection $q$ in $A^{**}$ is dominated by an
open projection $p$ in $A^{**}$, then
$q$ is compact in $p A^{**} p$ (viewed as the
second dual of the HSA supported by $p$),  if and only if $q$ is compact in $A^{**}$.
\end{corollary}

The word `intrinsic' is used to describe the next result because, in contrast to Theorem
\ref{1Ur} both projections $p, q$ are in $A^{**}$, and no containing 
$C^*$-algebra $B$  need be  mentioned in the
statement.

\begin{theorem} \label{bnc26}   {\rm (Intrinsic noncommutative  Urysohn lemma for 
Jordan operator algebras.) } \  Let  $A$ be a Jordan 
operator algebra.    Whenever a compact projection $q$ in $A^{**}$ is dominated by an
open projection $p$ in $A^{**}$, then there exists
$b \in \frac{1}{2} {\mathfrak F}_A$ with $q = qbq, b = pbp$.    Moreover, $q \leq u(b) \leq s(b)
\leq p$.
\end{theorem}

\begin{proof}  Just as in the proof of
\cite[Theorem 2.6]{BNII}, but using our Corollary \ref{bnc25} in place of the variation from \cite{BNII} used there.  The last statement follows from the
associative operator algebra case, since it may be viewed as a statement
in the operator algebra generated by $b$.    \end{proof}

By taking $n$th roots of the element $b$ in the last proof we may assume that $b$ is
`nearly positive' in the the sense of \cite{BRord} (c.f.\ Theorem 4.1 there).

\section{Peak projections and peak interpolation}
 \label{Peak}

We refer the reader to \cite{Bnpi} for a survey on noncommutative peak interpolation and peak projections 
for associative operator algebras.

As in \cite[Section 3]{BNII} if $A$ is  Jordan
operator algebra and $x \in {\rm Ball}(A)$, then we say that  $x$
{\em peaks} at a nonzero projection $q \in A^{**}$,
or that $q$ is a {\em peak projection} for $A$,
 if $A$ is a closed Jordan subalgebra of
a $C^{\ast}$-algebra $B$ and $x$ peaks at $q$ in the $C^*$-algebraic sense
at the start of \cite[Section 3]{BNII} with respect to $B$:
for example there exists a $x \in {\rm Ball}(B)$ such that $qxq = q$ (which is equivalent to
$xq = q$ since $x \in {\rm Ball}(A)$), and $\varphi(x^*x) < 1$ for all $\varphi \in Q(B)$
with $\varphi(q)  =0$.   
This is essentially Hay's definition
of a peak projection \cite{Hay}, and it forces 
$q$ to be compact with respect to that $C^*$-algebra $B$ as explained
at the start of \cite[Section 3]{BNII},
and hence also compact with respect to $A$ as we said
above Corollary \ref{bnc23}.
 In \cite[Lemma 3.1]{BNII} some equivalent conditions
are given, including that the tripotent mentioned towards the end
of the introduction of \cite{BNII}, namely the weak* limit of $x(x^*x)^n$,
is a projection.   The latter projection is $q$ in the notation at the start of
this paragraph, and $x^n \to q$ weak*.
If $q$ is a peak for $x$ as above then we write $q = u(x)$.
  This latter definition is independent of the 
particular containing $C^{\ast}$-algebra $B$.  Indeed we have:

\begin{proposition}  \label{peakare}  If $A$ is a Jordan operator algebra,
then the peak projections for $A$ are the weak* limits of $a^n$ for
$a \in {\rm Ball}(A)$, in the case such weak* limit exists.
\end{proposition} 

\begin{proof}   This is as in \cite[Lemma 1.3]{BRII}, the main part of which is
a computation in the associative
operator algebra generated by $a$.   \end{proof}

It follows from the last results that if  $q \in A^{**}$  and  
 $q$ is a peak projection for $A$ then  $q$ is a peak projection for $A^1$
(we shall prove the converse of this in Proposition \ref{BRII64}), and for 
  $x\in  {\rm Ball}(A)$ we have that $x$ peaks at $q$ with respect 
to $A$    if and only if $q$ is a peak for $x$ with respect 
to $A^1$ (c.f.\ \cite[Corollary 3.2]{BNII}).    
If $a \in \frac{1}{2} {\mathfrak F}_A$ then $(a^n)$ converges weak*
to a projection $q$ at which $a$ peaks.   More generally all the statements
in \cite[Corollary 3.3]{BNII} will be true in the Jordan
operator algebra case since they follow from the same
statements in the (associative) operator algebra generated by $a$. 
As we said in the introduction if $x \in {\rm Ball}(A)$ peaks at  $u(x)$ 
and $A$ is unital  then we have the relation $u(x) = 1 - s(1-x)$.  

The following is a variant on  \cite[Proposition 5.5]{Hay}, and may be new in its present generality even for 
associative operator algebras:

\begin{corollary} \label{cipeak}  If $A$ is a Jordan operator algebra,
then the infimum of any countable collection of peak projections for $A$, if it is nonzero,
 is a peak projection
for $A$.    Indeed if $(a_n)$ is a sequence of contractions in $A$ each of which have a nonzero 
peak projection $u(a_n)$ then $\wedge_n \, u(a_n) = u(a)$ where the latter
is the peak of $a = \sum_n \,
\frac{a_n}{2^n}$.    \end{corollary}
  
\begin{proof}   If  $A$ is a closed Jordan subalgebra of a unital $C^*$-algebra $B$, then we may replace
$A$ by $B$.   Thus this corollary is a result about (unital) $C^*$-algebras, and  then the result
follows from the argument 
for \cite[Proposition 1.1]{BNII}, but with Ball$(A^*)$ replaced by the state space.   Note that 
if $\psi(u(a_n)) = 1$ for all $n$ if  and only if $\psi(\wedge_n \, u(a_n) ) = 1$.    Indeed the one direction of this
is obvious.   For the other    direction, if $\psi(u(a_n)) = 1$ for all $n$, viewing $\psi$ as a weak* continuous
state on the bidual, it is well known that $\psi(\wedge_n \, u(a_n) ) = 1$.  The argument we are following
then gives $\wedge_n \, u(a_n) = u(a)$.    
\end{proof}

\begin{lemma} \label{Hayhm}   Let $A$ be a Jordan operator algebra, let $q$ be  a projection
 in $A^{**},$ and let $a \in  {\rm Ball}(A)$ with $qaq=q$.   If $a$ peaks at $u(a)$,
 then $q\leq u(a).$       If in addition $A$ is unital then $q\leq u(\frac{1}{2}(1+a))$ even if $a$ has no peak.
\end{lemma}  \begin{proof} As we said $qaq=q$
is equivalent to $q = qa$, so that $qa^n=q = qa^nq$, and in the limit
$q u(a) q = q$.    So $q\leq u(a).$   For the last statement replace $a$ by $\frac{1}{2}(1+a)$ in
the last argument,   $\frac{1}{2}(1+a)$ does have a peak
by a fact above Corollary \ref{cipeak}  since it is
in $\frac{1}{2} {\mathfrak F}_A$.  \end{proof}

Some of the following generalization of \cite[Theorem 3.4]{BNII} was found
by the first author and Zhenhua Wang.

\begin{theorem}  \label{inco}
If $A$ is an approximately unital Jordan operator algebra, then
\begin{enumerate} 
\item [{\rm (1)}]
  A
projection $q \in A^{**}$ is  compact  if and only if it is  a decreasing
limit of peak projections for $A$.  This is equivalent to $q$ being the infimum
of a set of peak projections for $A$.
\item [{\rm (2)}] 
If $A$  is separable, then the compact
projections in $A^{**}$ are precisely the peak projections.
\item  [{\rm (3)}]  
A projection in $A^{**}$
is a peak projection in $A^{**}$
 if and only if it is of form $u(a)$
for some $a \in \frac{1}{2} {\mathfrak F}_{A}$.
\end{enumerate}
\end{theorem}  \begin{proof}  Suppose that $A$ is a Jordan subalgebra
of a $C^*$-algebra $B$.

(2) \ We said above that peak projections
are compact.   Conversely suppose that $q$ is compact,
with $q = q x q$ for $x \in {\rm Ball}(A)$.
Then $1-q$ is open in $(A^1)^{**}$.
If $A$, and therefore also $A^1$, is separable, 
then $1-q = s(x)$ for some $x \in \frac{1}{2}{\mathfrak F}_{A^1}$
by \cite[Corollary  3.21]{BWj} (one may also have to use e.g.\ 
Lemma 3.7 there).   Thus if $z = 1-x \in {\rm Ball}(A^1)$ then by the fundamental relation
between $s(\cdot)$ and $u(\cdot)$,  
 $q = u(z) = u(z) z u(z)$.
Let $a = z \circ x \in {\rm Ball}(A)$. Working in $B^{**}$ we have
$$qa=q\frac{zx+xz}{2}=\frac{qx+qz}{2}=q.$$
For any compact projection $p \leq 1-q$  in $B^{**}$ we have
$$\Vert pa\Vert = \Vert p(\frac{zx+xz}{2}) \Vert
\leq \| \frac{pz}{2} \| + \| \frac{px}{2} \| < \frac{1}{2} + \| \frac{px}{2} \|
\leq 1$$
by \cite[Lemma 3.1]{BNII}.
 So $q = u(a)$ by that same lemma.

(1) \ The one direction of the first `iff' follows from
Corollary \ref{bnc23}.  For the other, if 
$q \in A^{**}$ is compact, with $q = q x q = qx$ for some $x \in {\rm Ball}(A)$,
 Then working in $B^{**}$ we can use the arguments 
in the proof of \cite[Theorem 3.4 (1)]{BNII} to produce elements
$z_t \in {\rm Ball}(A)$ with $q \leq u(z_t) \leq q_t$, and $u(z_t) \leq u(x)$,
 and $(u(z_t))$ decreasing projections.  Let $a_t = z_t \circ x$.  Then   
$$u(z_t) x = u(z_t)^* x = u(z_t)^* u(z_t) u(x)^* x  = u(z_t)^* u(z_t) u(x)^* u(x)
u(x)^* x = u(z_t)^* u(x)$$
since $u(z_t) \leq u(x)$ (we are using basic facts about the tripotent $u(x)$ 
that were used in the proof we are copying).
Hence $u(z_t) x = u(z_t)^* u(z_t) u(z_t)^* u(x) = u(z_t)$.
Using this and similar arguments it is then easy to see
 that $u(z_t) a_t = u(z_t) \frac{z_t x+xz_t}{2}  = u(z_t)$.  
As in the end of the proof in (1) for any compact projection $p \leq 1-u(z_t)$ 
we have $\| pa\Vert < 1$.  We may now finish as in the proof we are following.

(3) \ The one direction follows from a
fact stated after Lemma \ref{Hayhm}, and for the other direction again 
we may follow the proof of \cite[Theorem 3.4 (3)]{BNII}.  The main change 
is that we view the product $d = rb$ there in $B$.  The argument shows
that $d$ peaks at $q$.   Then $x = dr = rbr \in A$
and the argument shows   
that $x$ peaks at $q$.  The algebra $D$ there is an associative 
operator algebras so the rest of the argument for \cite[Theorem 3.4 (3)]{BNII}
is unchanged until the final line, which clearly holds too
in the Jordan case.     \end{proof}

The following is the Jordan generalization of \cite[Proposition 6.4]{BRII}:

\begin{proposition} \label{BRII64}  
Let $A$ be a nonunital Jordan operator algebra.
\begin{itemize}
\item [(1)] If $q$  is a nonzero projection in $A^{**}$ and $q = u(x)$   
for some $x \in {\rm Ball}(A^1)$, then $q$  is a peak projection with respect to 
$A$ (thus $q = u(a)$
for some $a \in {\rm Ball}(A)$).
\item [(2)] If $A$ is separable then the compact projections in $A^{**}$
 are precisely the peak projections in $A^{**}$ (the projections $u(x)$
for some $x \in A$). \end{itemize}   
\end{proposition} \begin{proof} (1) \ 
As in the proof of \cite[Proposition 6.4]{BRII},
 $u(x)$ is compact in $A^{**}$, hence by
Theorem \ref{brii62} 
 there exists an element $b \in {\rm Ball}(A)$
with $q = qbq$.
As in  the last lines of \cite[Theorem 3.4 (2)]{BNII} (see also
the proof of Theorem 3.4 (3) there),
but  according to how that proof was amended for  Theorem \ref{inco},
we have $q = u((1-b) \circ b)$, and $a = (1-b) \circ b \in {\rm Ball}(A)$.

(2) \ Follows from (1) as in \cite[Proposition 6.4]{BRII}.     \end{proof}  
  
\begin{corollary}  \label{jmet}  If $a, b \in {\mathfrak r}_{A}$
for a Jordan operator algebra $A$, and if  $s(a)$ and $s(b)$ commute then their infimum   is
of form $s(c)$ for some $c \in \frac{1}{2} {\mathfrak F}_{A}$.  Similarly,
the supremum of two commuting
peak projections
 in $A^{**}$,
 is a peak projection  
 in $A^{**}$. 
\end{corollary}    

\begin{proof}    We may assume that $a, b \in \frac{1}{2} {\mathfrak F}_{A}$
by \cite[Lemma 3.7]{BWj}.  Then the  proof of the first assertion is just as in \cite[Corollary 3.5]{BNII},
but replacing the appeal to \cite{BRI} to an appeal to \cite[Corollary 3.21]{BWj}.  

There is a mistake in the argument for the second assertion in \cite[Corollary 3.5]{BNII}.  
This is fixed   by working in $A^1$ and noting that
$$u(a) \vee u(b) = 1 - (u(a)^\perp \wedge u(b)^\perp) = 1 - s(1-a) \wedge s(1-b)
= 1 - s(c)$$ for some $c \in \frac{1}{2} {\mathfrak F}_{A^1}$ 
by the first assertion. This equals 
$u(1-c)$ by the relation  $u(x) = 1 - s(1-x)$ mentioned earlier.  By Proposition \ref{BRII64} (1) 
we have $u(a) \vee u(b) = u(1-c)$ is a peak projection 
 in $A^{**}$.
 \end{proof}

Propositions 4.1,  4.3, 4.4 in \cite{BNII} on weak* closed faces in the state  
space generalize easily to the Jordan operator algebra case, using the 
fact that states on $A$ can be extended to states on the containing $C^*$-algebra
$B$.  One uses the same ideas, but suitably modified by replacing results cited
in \cite{BNII} by the matching results from
 \cite{BWj} or the 
present paper.

 We now turn to peak interpolation.   Our best result here  for associative operator algebras 
in \cite[Theorem 3.4]{Bnpi} works
in the Jordan case, but for deep reasons.  It is a vast generalization of a famous result of E. Bishop (and also generalizes Gamelin's  variant of Bishop's result in \cite[II.12.5]{Gam}).    Probably the best way to describe its value is to say that
it (and also our Urysohn lemmas)  allows us to build elements in $A$ which have certain prescribed properties or behavior, as we mentioned
in the introduction.  

\begin{theorem} \label{peakthang22}   Suppose that $A$ is a Jordan  operator algebra
(not necessarily approximately unital), a Jordan subalgebra of a unital $C^*$-algebra $B$.
Identify $A^1 = A + \Cdb 1_B$.  Suppose
 that  $q$ is a closed projection in $(A^1)^{**}$.
   If $b \in A$ with $b q= qb$ in $B$,
 and $q b^* b q  \leq q d$ in $B$ for an invertible positive $d \in B$ which commutes with $q$,
then  there exists an element $g \in A$ with $g q =  q g = b q$, and $g^* g \leq d$ (all these formulae
interpreted in $B$).   
\end{theorem}

\begin{proof}      
As in the proof of   \cite[Theorem 3.4]{Bnpi} let $\tilde{D}$ be the HSA supported by
$1-q$ in $A^1$, let $f = d^{-\frac{1}{2}}$, 
let $C$ be the closed Jordan subalgebra of $A^1$ generated by $\tilde{D}, b,$ and $1$, and let 
$$D = \tilde{D} \cap A \cap C = \{ x \in A \cap C : q \circ x = 0 \} \subset A .$$
    (Note that
$q \circ x = 0$ if  and only if $x = (1-q) x (1-q)$ as we said in the introduction.)   We have
$\tilde{D}^{\perp \perp} \subset q^\perp C^{\perp \perp} q^\perp$, and conversely
$q^\perp C^{\perp \perp} q^\perp \subset  q^\perp (A^1)^{\perp \perp} q^\perp  = \tilde{D}^{\perp \perp}$.
So $q^\perp \in \tilde{D}^{\perp \perp} = q^\perp C^{\perp \perp} q^\perp$.   Thus
$q \in C^{\perp \perp}$.    Working in $B^{**}$ the centered equation in the proof of \cite[Theorem 3.4]{Bnpi} 
holds (note that $q$ commutes with all terms mentioned there).   The expression $((C \cap A) f)^\perp$ in that 
proof is a subset of $(Cf)^*$, and $q$ commutes with elements of $Cf$ and $C$.   Set
$X = (C \cap A) f$.   Note $b \in X$.  If $(c_t)$ is a net in $C$
with weak* limit $q$, and if $d \in C \cap A$
and $\varphi \in X^\perp$, then  $c_t \circ d \in C \cap A$ so that $\varphi((c_t \circ d) f)= 0$.
Since 
$$\frac{1}{2}(c_t df + dc_t f) \to \frac{1}{2}(q d f + d q f) = q df q$$ we have 
$\langle q df q , \varphi \rangle = 0.$  That is, $q \varphi q = q \varphi \in  X^\perp$.  So as in the argument 
we are following,  
$X^\perp$ is invariant under (the $L$-projection of) multiplication by $q$ on $(Cf)^*$, which yields as in that proof that $Df= X \cap \tilde{D}f$ is an $M$-ideal in $X$.   The fact that $d(x,Df) = \| q x \| = \| q x q \|$ for
$x \in  X$, that this distance is achieved, and indeed the remainder of  the proof we are following, then follows verbatim as in that proof (no results from the present paper are needed in
these lines).  \end{proof}

Note that the analogous theorem \cite[Theorem 3.4]{Bnpi}  generalizes \cite[Corollary 2.2 and (most of) 
Theorem 5.1]{BRII}, and \cite[Lemma 2.1]{Bnpi}.   The same will be the case for Jordan operator algebras.  The following
is the Jordan version of \cite[Theorem 5.1]{BRII}, and also contains a Jordan generalization
of   \cite[Theorem  2.1]{BRII}:

\begin{corollary} \label{dist}    Suppose that $A$ is a Jordan  operator algebra
(not necessarily approximately unital), a Jordan subalgebra of a unital $C^*$-algebra $B$. Suppose
 that  $q$ is a closed projection in $(A^1)^{**}$.
   If $b \in A$ with $b q= qb$ in $B$ (or equivalently, 
$[q,b,q^\perp] = 0$ in $(A^1)^{**}$), then  $b$ achieves its distance to the intersection of $A$
with the HSA in $A^1$ supported by $1-q$ (this intersection is a HSA in $A$ if 
$A$ is approximately unital by Proposition {\rm \ref{commop}}).   If further
 $\| q b q  \| \leq 1$ then  there exists an element $g \in {\rm Ball}(A)$ with $g q =  q g = b q$.
 \end{corollary}

\begin{proof}     The final assertion follows from Theorem \ref{peakthang22}, and in the present case $f = 1$
in the proof of that result.   However we will give a different, second proof of this
in the next paragraph.    The assertion about achieving the distance from $b$ follows by combining the fact that the 
 distance from $b$ is achieved 
towards the end of the proof of Theorem \ref{peakthang22}, with the idea in the last three lines of the proof of \cite[Theorem 5.1]{BRII}. For the latter we need to know
that the distance from $b$ to  the HSA in $A^1$ supported by $1-q$ is $\| q b q \|$, but this
is assured by Lemma \ref{distH}.  

For a second proof of the final assertion, replace $J$ in the proof of  \cite[Theorem 5.1]{BRII} by $J = \{a \in A: q \circ a = 0\} = \{a \in A :a= q^\perp a q^\perp \}$.   The first part of that proof (first paragraph) showing $d(b,J) = \| qbq
\|$ (note $qb = qbq$ here) follows from our Proposition \ref{gamlemma1}. The next paragraph and a half leading to $d(b,D) = \| qbq \|$ follows easily, noting that $q$ commutes with $C$.  Indeed, in that proof as in 
the proof of Theorem \ref{peakthang22}, the argument involving 
 left multiplication by $q^\perp$ (which equals Jordan multiplication by $q^\perp$, and equals multiplication on both
sides by $q^\perp$) works, and gives an $M$-ideal.    The remaining three lines of that proof we have discussed in the 
first paragraph of the present proof, but to avoid using part of the
  proof of Theorem \ref{peakthang22}  use the formula $d(b,D) = \| qbq \|$ from a couple of lines
back.
\end{proof}

We are not sure if there is a Jordan variant of (Lemma 5.3 and) Theorem 5.4 in \cite{BRII}.
Indeed this proof seems very one-sided.   However since  Theorem 5.4 in \cite{BRII}  is a variant of Theorem 5.2 there,
and  since we generalized the latter to Jordan algebras in our
Theorem \ref{peakthang22}, we are not sure if it is worth searching out.      In any case this is one of the 
very  few results from the papers we are generalizing which seems obstinate.  

\section{Miscellaneous results in noncommutative topology} \label{sec9}

Turning to the final Section 6 in \cite{BRII}, we discussed Lemma 6.1 there after our 
Theorem  \ref{chcomp}.
The following is the Jordan variant of Theorem 6.2 there, and is a nonunital version of our earlier Theorems \ref{chcomp}
and \ref{inco} (1): 

\begin{theorem} \label{brii62}   If $A$ is a Jordan  operator algebra, a
closed Jordan subalgebra of a $C^*$-algebra $B$, and $q$ is a projection in $A^{**}$  then 
 the following are equivalent:
\begin{enumerate}
\item [(i)] $q$  is compact in  $B^{**}$ in the sense of Akemann,
\item [(ii)] $q$  is compact in  $A^{**}$ (that is, $q$ is a closed projection in $(A^1)^{**}),$ 
\item [(iii)]  $q$ is a closed projection in $(A^1)^{**},$ and there exists an element $a \in {\rm Ball}(A)$
with $q = qaq$,
\item [(iv)] $q$ is  a decreasing limit of peak projections for $A$.
\end{enumerate}
\end{theorem} \begin{proof}  We said after Theorem \ref{chcomp} that (i) and (ii) 
were equivalent.  That (iii) implies (iv) and (iv) implies (ii) is as in the proof of \cite[Theorem 3.4 (1) and (2)]{BNII}, 
 according to how those proofs were amended for  Theorem \ref{chcomp}.   That (ii) implies (iii) is as in the first paragraph of the 
proof of  \cite[Theorem 6.2]{BRII}, but using our Corollary \ref{dist} in the place where
 \cite[Theorem 5.1]{BRII} is invoked. \end{proof}

We remark that the Jordan version of Proposition 5.1 of \cite{BNII} clearly true with essentially the same proof.

\begin{corollary} \label{limba} Let $A$ be a Jordan operator algebra.
Then  $q\in A^{**}$ is compact if and only if  $q$ is a weak* limit of a net $(a_t)$ in ${\rm Ball}(A)$ with $q a_t q=q$ for all $t.$ \end{corollary}

\begin{proof}  ($\Leftarrow$) \ Note that $q^\perp (1-a_t) q^\perp = 1-a_t \to 1-q$ weak*,
so that $q^\perp$ is open in $(A^1)^{**}$, so that $q$ is closed there.  Hence $q$
is compact in $A^{**}$.

($\Rightarrow$) \ This follows as in \cite[Corollary 6.3]{BRII}, using Theorem \ref{brii62} (iv) and adding some $q$'s 
on the right of some of the expressions there.
\end{proof}

As on p.\ 1065 of \cite{BRII}, if $A$ is a 
non-approximately unital operator algebra, we  define 
a ${\mathfrak F}$-{\em peak projection} for $A$ to be $u(x)$, the weak* limit
of the powers $x^n$, for 
some $x \in \frac{1}{2} {\mathfrak F}_A$.   See 
the lines above Corollary \ref{cipeak}  for the fact that this
weak* limit exists and is a projection, which is 
nonzero if $\Vert x \Vert = 1$.  We define a
projection in $A^{**}$ to be ${\mathfrak F}$-compact if it is a decreasing 
limit of ${\mathfrak F}$-peak projections.  

\begin{proposition} \label{cohm}  If $A$ is any
non-approximately unital Jordan  operator algebra, then 
\begin{itemize}  
\item [(1)]  A projection in $A^{**}$  is
${\mathfrak F}$-compact in the sense above if  and only if it is a compact projection
with respect to the approximately unital Jordan  operator algebra 
 $A_H$ (the latter was defined above Corollary {\rm 4.2} in {\rm \cite{BWj}}).   
\item [(2)]  A projection in $A^{**}$  is an ${\mathfrak F}$-peak projection in the sense above if  and only if it is a peak projection
 for $A_H$.  
\item [(3)] If $A$ is separable then every 
${\mathfrak F}$-compact projection in $A^{**}$ is an ${\mathfrak F}$-peak projection.  
 \end{itemize} \end{proposition} \begin{proof}   As in \cite[Proposition 6.5]{BRII},
but using the fact from the proof of Corollary {\rm 4.2} in {\rm \cite{BWj}} that 
${\mathfrak F}_A = {\mathfrak F}_{A_H}$, and replacing appeals to \cite{BNII}
with the matching results in the present paper and \cite{BWj}.   
\end{proof}

We will treat \cite[Theorem 6.6]{BRII} below in our discussion of the 
Urysohn lemmas from \cite[Section 4]{BRord}.
Turning to the latter paper, the noncommutative topology there begins in earnest with Lemma 3.6 there.  The
reader is reminded there that a {\em $\sigma$-compact} projection in $B^{**}$ 
for a $C^*$-algebra $B$,
is an open projection $p \in B^{**}$ which is 
the supremum (or weak* limit)
of an increasing sequence in $B_+$.      The Jordan 
variant of  \cite[Lemma 3.6]{BRord} is as follows:

\begin{lemma} \label{cssigc}  If $A$ is a closed Jordan subalgebra of a $C^*$-algebra $B$, and if $p$ is 
an open projection in $A^{**}$ then the following are equivalent:
\begin{itemize} \item [(i)] $p$ is the support projection of
a HSA in $A$
 with a countable  cai.
\item [(ii)]  $p$ is $\sigma$-compact  in $B^{**}$ in the sense above.
\item [(iii)]  $p$ is the support projection of
a HSA  in $A$ of the form 
$\overline{xAx}$ for some $x \in {\mathfrak r}_A$.  That is, $p = s(x)$ 
for some $x \in {\mathfrak r}_A$. 
\item [(iv)]  There is a
sequence $x_n \in {\mathfrak r}_A$ with $x_n = p x_n p \to p$ weak*. 
\item [(v)] $p$ is the support projection of a strictly real  
positive element (in the sense defined after Lemma  {\rm 3.11} in {\rm \cite{BWj}}) of the hereditary subalgebra defined by $p$.
\end{itemize} 
If these hold then the sequence $(x_n)$
in {\rm (iv)} can be chosen with $(x_n + x_n^*)$ increasing,
and the $x_n$, and the element $x$ in {\rm (iii)}, can be chosen to be
in $\frac{1}{2} {\mathfrak F}_A$.
  \end{lemma}  \begin{proof}   This follows the lines of proof of  \cite[Lemma 3.6]{BRord}
replacing appeals to various facts 
with the matching results in the present paper and \cite{BWj}, with the following exceptions:
In the proof of (v) one needs to  show that $D = \overline{x A x}$ of course, but this is similar
to the proof we are copying (there is a typo in that proof, $xAD$ should read $x D D$).  
When proving (iv) implies (iii) we replace right ideals by HSA's, 
and  let $J = \{ a \in A : a = pap \}$, a HSA by definition.   Note that this contains
$\overline{x_n A x_n}$ for each $n$.   We replace each $x_n$ by ${\mathfrak F}(x_n)$,
which does not change  $\overline{x_n A x_n}$.    Thus assume $x_n \in \frac{1}{2} {\mathfrak F}_A$.
Then   
$\overline{x_n A x_n} \subset \overline{x A x} \subset J$ where
$x = \sum_n \, \frac{x_n}{2^n}$ by
\cite[Corollary 3.17]{BWj}.     By that result we also have $s(x) x_n s(x)= x_n$, so that in the limit
$s(x) p s(x) = p$.   Thus $p \leq s(x)$ and so $\overline{x A x} = J$.

Finally, in the equivalence of (i) and (ii), we will have $I = JBJ^*$ by Theorem \ref{bhn29}, where $I$ and $J$
are the HSA's supported by $p$ in $B$ and $A$ respectively.   So a countable cai for $J$ will be a countable cai
for  $I$
using e.g.\ \cite[Lemma 2.1.6]{BLM}.   Conversely if $I$ has a countable cai, we leave it as an exercise to modify
the proof in   \cite[Lemma 3.6]{BRord} that 
$J$ has a countable cai.
\end{proof}

If $A$ is a Jordan operator algebra then
an open projection $p \in A^{**}$ is called {\em $\sigma$-compact} with respect to $A$ if
it satisfies the equivalent conditions in the previous result.  This is all used in the 
`strict Urysohn lemma' below the next result (part (1) of which also is a Urysohn lemma):

\begin{theorem} \label{urys}  Let $A$ be an approximately unital Jordan 
subalgebra of a  $C^*$-algebra $B$, and let
$q \in A^{\perp \perp}$ be a compact projection.
\begin{itemize}
\item [(1)]  If $q$ is dominated by an open projection  $u \in B^{**}$
then for any $\epsilon > 0$,
there exists an $a \in 
\frac{1}{2} {\mathfrak F}_A$ with $q \circ a =  q$,
and $\Vert a (1-u) \Vert < \epsilon$ and $\Vert (1-u)  a \Vert < \epsilon$.
Indeed this can be done with in addition $a$ nearly positive
(thus the numerical range of $a$ in $B^1$ within
a horizontal cigar centered on the line segment $[0,1]$ 
in the
$x$-axis, of height $< \epsilon$).  
\item [(2)]  $q$ is a weak* limit of a net $(y_t)$ of nearly
positive elements in $\frac{1}{2}
\, {\mathfrak F}_A$ 
with $y_t \circ q =  q$.
\end{itemize} \end{theorem}
 \begin{proof}    (2) \  First assume that $q = u(x)$
  for some $x \in \frac{1}{2} {\mathfrak F}_A$. As in
the proof of \cite[Theorem 4.2]{BRord} we may replace $A$ by the commutative 
associative operator algebra ${\rm oa}(x)$.
Then we may simply appeal to the argument there
to obtain what is stated in (2).   

Next, for an arbitrary compact projection $q \in  A^{\perp \perp}$, by 
Theorem \ref{inco} 
there exists a net $x_s \in \frac{1}{2} {\mathfrak F}_A$  with $u(x_s) \searrow q$.
By the case in the last paragraph there exist nets $y_t^s \in \frac{1}{2} \, {\mathfrak F}_A$
with $u(x_s) y_t^s \, u(x_s) = u(x_s)$, and $y_t^s \to u(x_s)$ weak*.
Then $$q y_t^s q = q u(x_s) \, y_t^s \, u(x_s) \, q = q u(x_s) \, q = q,$$ for each $t, s$.  
As in the proof we are following the $y_t^s$ yields
 a net weak* convergent to $q$.

(1) \ The first assertion  follows from (2) as in the proof of the approximately 
unital case of  \cite[Theorem 4.2 (1)]{BRord}.  Namely 
by substituting the net $(y_t)$ from (2) into
the proof of Theorem \ref{1Ur} one obtains the
first assertion of (1).  The other  assertions of (1)
are as in \cite[Theorem 4.2 (1)]{BRord}.
  \end{proof}  
 
   The Jordan
variant of  \cite[Lemma 4.4]{BRord} is as follows:

\begin{lemma} \label{sigcp} Suppose that $A$ 
is an approximately
unital Jordan operator algebra, with $e = 1_{A^{**}}$, that $q \in A^{**}$ is compact,
and that $p = q^{\perp} = e-q$ is $\sigma$-compact in $A^{**}$. Then $q$ is a 
peak projection for $A$, indeed $q = u(x)$ for some  $x \in \frac{1}{2} {\mathfrak F}_A$.  
\end{lemma}

\begin{proof} 
This follows by the idea of proof of \cite[Lemma 4.4]{BRord}
suitably modified by using the variant found in the argument for our
Theorem \ref{inco} (3).  Namely one shows in the notation of those proofs
$d = rb \in B$ satisfies $u(d) = q$, and $w = dr = rdr \in A$
satisfies $u(w) = u(dr) = q$.
The remainder of the proof is as for \cite[Lemma 4.4]{BRord}
but using matching results from \cite{BWj} or the present paper
whenever various other results are invoked. \end{proof}  
The remark in the last line of this proof also applies to the proofs of the 
following variants of \cite[Corollary 4.5 and Theorem 4.6]{BRord}.

\begin{corollary} \label{ux} Suppose $A$ is a closed Jordan subalgebra of a $C^*$-algebra $B$.
If a peak projection for $B$ lies in $A^{\perp \perp}$ then it is also a peak projection for $A$.  
\end{corollary}

   Conversely, any peak projection for $A$ is also a peak projection for $B$.
We also remark that a different proof of Corollary \ref{ux} was given in \cite{blueda}.  In the Jordan situation this
proof would go as follows: a peak projection $q$ for $B$ is a limit of a decreasing sequence of positive elements of $B$ (namely $(b^n)$ if
a positive contraction $b$ peaks at $q$).   So $q^\perp$ is a limit of an increasing sequence in $B^+$.
Thus $q^\perp$ is $\sigma$-compact in the sense above Lemma \ref{cssigc}, and by that result 
$q^\perp = s(x)$ 
for some real positive $x$ which we may rechoose in $\frac{1}{2} {\mathfrak F}_{A^1}$.      One may then follow the last lines of the proof of Corollary \ref{jmet} to see that $q$ is a  peak projection for $A$.

\begin{theorem} \label{sul} {\rm (A strict noncommutative Urysohn lemma for 
Jordan operator
algebras)} \ Suppose that $A$ is any Jordan operator algebra
and that 
$q$ and $p$ are respectively compact   
and open projections in $A^{**}$ with $q \leq p$, and
$p - q$ $\sigma$-compact (note that the latter is automatic if $A$ is separable). 
Then there exists $x \in \frac{1}{2} {\mathfrak F}_{A}$ such that $qx q = q$ and $p x p =
 x$,
and such that $x$ peaks at $q$ (that is, $u(x) = q$) and $s(x) = p$, and $1-x$ peaks at
$1-p$ with respect to $A^1$
(that is, $u(1-x) = 1-p$). The latter identities imply
that $x$ is real strictly positive
in the hereditary subalgebra $C$ associated with $p$, and $1-x$ is real strictly positive
in the hereditary subalgebra in $A^1$ associated with $1-q$. Also, $s(x(1-x)) = p-q$,
so that $x(1-x)$ is real strictly positive
in the hereditary subalgebra in $A$ associated with $p-q$. We can also have $x$ `almost
positive', in the sense that 
if $\epsilon > 0$ is given one can choose $x$ as above but also satisfying 
${\rm Re} (x) \geq 0$ and $\Vert x - {\rm Re} (x) \Vert < \epsilon$.  
\end{theorem}

This is proven just as in the proof of the strict noncommutative Urysohn lemma 
in \cite[Theorem 4.6]{BRord}.  

The Jordan variant of the Remark 1 after  \cite[Theorem 4.6]{BRord} also seems
to hold, simply using matching results from \cite{BWj} or the present paper
whenever various other results are invoked. 

We also obtain the `lifting of projections' application as in \cite[Corollary 4.8]{BRord}:

\begin{corollary} \label{liftp}  Let $A$ be a Jordan operator algebra containing a closed two-sided ideal $J$ with a countable cai, or
equivalently, with a $\sigma$-compact  support
projection $p$. Also, suppose that $q$ is a projection in $A/J$. Then there exists an 
almost
positive $x \in \frac{1}{2} {\mathfrak F}_A$ such
that $x + J = q$.
Also, the peak 
$u(x)$ for $x$ equals the canonical copy
of $q$ in $(1-p) A^{**} (1-p)$. \end{corollary}  

\begin{proof}  One first needs to check that the Jordan variant of \cite[Lemma 4.7]{BRord}
holds with essentially unchanged proof, but citing our
Proposition \ref{bri62} instead of the references to \cite{BRI}. Then  one copies the proof of
 \cite[Corollary 4.8]{BRord}.   There may be a typo in the reference to \cite{BRI} there, that probably should
be to Proposition 6.1 there, a result that we generalized to Jordan algebras in \cite[Proposition 3.28]{BWj}.
Since $p$ is a central projection the proof we are copying proceeds without issue, merely
replacing $yr,xr$ and $x(p+r)$ there by $ryr, rxr$
and $(p+r)x(p+r)$.
  \end{proof}  

Proposition 4.9 in \cite{BRord} is also valid for approximately
unital Jordan operator algebras.

\begin{remark} There may be a gap in the proof of Theorem 4.10  in \cite{BRord}, we do not follow why in the 
fifth line of the proof $D$ is approximately unital if $I$ is not approximately unital.    Also, 
  Lemma 4.11 in \cite{BRord} seems to have an issue in the Jordan case, since we do not know 
at the time of writing if the quotient of a Jordan operator algebra by a closed ideal $D$  is a Jordan operator algebra (and actually believe this is false).   However the latter
is true if $D$ is approximately unital by \cite[Proposition 3.27]{BWj}, so that \cite[Lemma 4.11]{BRord}
is certainly valid for unital Jordan operator algebras $C, D$ and closed Jordan subalgebras $A, B$ provided that the 
kernel of the $q$ there is approximately unital, and provided we drop the word `completely'.

The Tietze theorem application Theorem 4.12 in \cite{BRord} of 
\cite[Lemma 4.11]{BRord} is thus endangered in the Jordan case.   However we remark that 
where we actually apply Lemma 4.11 in \cite{BRord} there, the ideal $D$ we are quotienting by
is actually an $M$-ideal, hence proximinal, so the proof looks like it may be salvageable.   We hope
to visit these topics at a later time.   \end{remark}

Finally, we remark on  the Jordan algebra case of some results from \cite[Section 2]{blueda}.
We have already generalized the two lemmas in \cite[Section 2]{blueda} in Corollary \ref{cipeak}
and  Corollary \ref{ux} and the remark after it.     Theorem 2.5 of  \cite{blueda} holds if $A$ is a Jordan subalgebra of $M$, with the same proof. The proof of the following result is very similar to 
the proof in \cite[Section 2]{blueda}.  We leave the adaption as an exercise (since this is a selfadjoint result one would
replace real positive elements by elements that are positive in the usual sense).   

 \begin{proposition} \label{cipeak3}  If $B$ is a JC*-algebra and $q$ is a projection in $B^{**}$, the following are equivalent: \begin{itemize} 
\item  [(i)]  $q$ is a peak projection with respect to $B$, indeed there exists $b \in {\rm Ball}(B_+)$ with $u(b) = q$.
\item   [(ii)]    $q$ is a compact projection which is  $G_\delta$ (that is, an infimum in  $B^{**}$ of a sequence $(p_n)$ of open projections).    
\item    [(iii)]    $q$ is the weak* limit of a decreasing sequence from $B_{\rm sa}$.  \end{itemize} 
\end{proposition}

\bigskip

{\em Acknowledgement:}  A few results here about compact and peak projections (which are adaptions
of arguments from \cite{BHN, BNII}) were found 
by the first author and Zhenhua Wang, and will appear also in Wang's thesis.   We thank him for permission
to include them here for completeness of our exposition.  We also thank him and the referee for spotting some typos.

\vspace{\baselineskip}


\begin{thebibliography}{10}

 

 \bibitem{Ake2}   C. A. Akemann, {\em Left ideal structure of $C\sp*$-algebras,}  J.\ Funct.\ Anal.\
\textbf{6}  (1970), 305-317.

\bibitem{APfaces} C. A. Akemann and 
G. K. Pedersen, {\em 
Facial structure in operator algebra theory,}
 Proc. London Math. Soc.  {\bf 64} (1992),  418--448. 


\bibitem{AS}   J. Arazy and B. Solel, {\em 
Isometries of nonselfadjoint operator algebras,} 
J. Funct. Anal. {\bf 90} (1990),  284--305. 

\bibitem{BBS}  C. A. Bearden, D. P. Blecher and S. Sharma, {\em On positivity and roots in operator algebras,}  Integral Equations Operator Theory {\bf  79} (2014), no. 4, 555--566.


\bibitem{Bnpi}  D.\  P.\ Blecher, {\em Noncommutative peak interpolation
revisited}, Bull.\ London Math.\ Soc.\ {\bf 45} (2013),
1100-1106.



\bibitem{BHN}  D. P. Blecher, D. M. Hay, and
M. Neal, {\em Hereditary subalgebras of operator algebras,} J.\
Operator Theory {\bf 59} (2008), 333-357.


\bibitem{blueda}  D. P. Blecher and L. E. Labuschagne, {\em Ueda's peak set theorem for general von Neumann algebras,} Preprint (2016),  to appear Trans. Amer. Math. Soc.

\bibitem{BLM}  D. P. Blecher and 
C.  Le Merdy, {\em Operator algebras and their modules---an
operator space approach,} Oxford Univ.\  Press, Oxford (2004).

\bibitem{BNI} D. P. Blecher
and M. Neal, {\em Open projections in operator algebras I: Comparison theory,} Studia Math.\ {\bf 208} (2012), 117--150.


\bibitem{BNII} D. P. Blecher and M. Neal, {\em Open projections in operator algebras II: Compact projections,} Studia Math.\ {\bf 209} (2012), 203--224.

\bibitem{BNmetric} D. P. Blecher
and M. Neal, {\em Metric characterizations II,} Illinois J.\  Math.  {\bf 57} (2013), 25--41.
 
\bibitem{BNp}   D. P. Blecher
and M. Neal, {\em Completely contractive projections on operator algebras,}  Pacific J. Math.   {\bf 283-2} (2016), 289--324.

\bibitem{BNjp} D. P. Blecher
and M. Neal, {\em Contractive projections on Jordan operator algebras (tentative title),} 
manuscript in preparation.

 \bibitem{BOZ}   D. P. Blecher and N. Ozawa, {\em Real 
positivity and approximate identities
in Banach algebras}, Pacific J. Math.  {\bf 277} (2015), 1--59.

\bibitem{BRI}  D. P. Blecher and C. J. Read, {\em  Operator algebras with contractive approximate identities,}
J. Funct.\ Anal.\  {\bf 261} (2011), 188--217.

\bibitem{BRII}  D. P. Blecher and C. J. Read, {\em  Operator algebras with contractive approximate identities II,} J. Funct.\ Anal.\  {\bf 264} (2013), 1049--1067.

\bibitem{BRord}  D. P. Blecher and C. J. Read, {\em   Order theory  and interpolation in operator algebras,}   Studia Math. {\bf 225} (2014), 61--95.

  \bibitem{BRoyce}  D. P. Blecher and M. Royce, {\em   Extensions of operator algebras I,} 
J. Math.\ Analysis and Applns.\  {\bf  339}  (2008), 1451--1467.

\bibitem{BRS}  D. P. Blecher, Z-J. Ruan, and A. M. Sinclair, {\em   A characterization of operator algebras}, J.\ Funct.\ Anal.\  {\bf 89} (1990), 188--201.


\bibitem{BWj}  D. P. Blecher and Z. Wang, {\em Jordan operator algebras: basic theory,}
Preprint 2017, to appear Math.\ Nachrichten.

\bibitem{BWj2}  D. P. Blecher and Z. Wang, {\em Jordan operator algebras: supplemental results,} (tentative title)
Preprint 2018, in preparation.

\bibitem{Brown} L.  G. 
Brown, {\em Semicontinuity and multipliers of $C^???$-algebras,} Can.\ J.\ Math. {\bf 40} (1988),  865--988. 


\bibitem{BFT}  L. J. Bunce, B. Feely, and R. M. Timoney, {\em Operator space structure of JC$^???$-triples and TROs, I,} Math. Z. {\bf 270} (2012), 
961--982. 

\bibitem{BT}  L. J. Bunce and R. M. Timoney, {\em Universally reversible JC???-triples and operator spaces,} J. Lond. Math. Soc. {\bf  88} (2013), 271--293.  


\bibitem{Rod}  M. Cabrera Garcia and A.  Rodriguez Palacios, {\em 
	Non-associative normed algebras. Vol. 1,}
The Vidav-Palmer and Gelfand-Naimark theorems. Encyclopedia of Mathematics and its Applications, 154. Cambridge University Press, Cambridge, 2014.

\bibitem{FPP0}  F. J. Fern\'andez-Polo and A. M. Peralta, {\em Compact tripotents and the Stone–Weierstrass Theorem for $C^∗$-algebras and JB∗-triples,}
J.\ Operator Theory {\bf 58} (2007), 157--173.
 
\bibitem{FPP}  F. J. Fern\'andez-Polo and A. M. Peralta, {\em Non-commutative generalisations of Urysohn's lemma and hereditary inner ideals,} 
J.\ Funct.\ Anal.\ {\bf 259} (2010), 343--358.


\bibitem{Gam}   T. W. Gamelin, {\em Uniform Algebras,} Second edition, Chelsea, New York, 1984.

\bibitem{Ham}  M. Hamana,  {\em Triple envelopes and Silov boundaries of operator spaces,} Math. J. Toyama Univ. {\bf 
22} (1999),  77--93.

\bibitem{HS}  H. Hanche-Olsen and E. Stormer, {\em  Jordan operator algebras,} 
Monographs and Studies in Mathematics, 21,  Pitman (Advanced Publishing Program), Boston, MA, 1984.   

\bibitem{Har2}  L. A. Harris, {\em  A generalization of $C^*$-algebras,} Proc. Lond. Math. Soc. {\bf 42}  (1981),  331--361.

\bibitem{Hay}  D. M. Hay, {\em Closed projections and peak interpolation for operator algebras,}
  Integral Equations Operator Theory  {\bf 57}  (2007),  491--512.  

\bibitem{HWW}  P. Harmand, D. Werner, and W. Werner,
{\em $M$-ideals in Banach spaces and Banach algebras,}    
 Lecture Notes in Math.,  1547, Springer-Verlag, Berlin--New York, 1993. 

\bibitem{HN}  G. Horn and E. Neher, {\em 
Classification of continuous JBW???-triples,}  
Trans. Amer. Math. Soc. {\bf 306} (1988), 553--578. 

\bibitem{KP} M. Kaneda and V. I. Paulsen, {\em 
Quasi-multipliers of operator spaces,} J.\ Funct.\ Anal.\ {\bf 217} (2004), 347--365. 

\bibitem{Kirch} E.  Kirchberg, 
 {\em On restricted perturbations in inverse images and a description of normalizer algebras in $C^*$-algebras,}  J. Funct. Anal. {\bf 129} (1995), 1--34. 

\bibitem{Lin}  H. Lin, {\em Cuntz semigroups of $C^*$-algebras of stable rank one and projective Hilbert modules}, Preprint 
arXiv:1001.4558 (2010).



\bibitem{Mc} K. McCrimmon, {\em A taste of Jordan algebras,}  
Universitext. Springer-Verlag, New York, 2004.


\bibitem{N} M.  Neal, {\em Inner ideals and facial structure of the quasi-state space of a JB-algebra,}
 J. Funct. Anal.\ {\bf  173} (2000), 284--307.

\bibitem{N2} M.  Neal, E. Ricard and B. Russo, {\em Classification of contractively complemented Hilbertian operator spaces,} J.\ Funct.\ Anal.\  237 (2006), 589--616.


\bibitem{NR03}  M. Neal and B. Russo, {\em Contractive projections and operator spaces}, Trans. Amer. Math. Soc. {\bf 355} (2003), 2223--2262. 

\bibitem{NeRu} M. Neal and B. Russo, {\em Operator space
characterizations of
C*-algebras and ternary rings}, Pacific J. Math. {\bf 209} (2003), 339--364.

\bibitem{NR14} M. Neal and B. Russo, {\em A holomorphic characterization of operator algebras}, Math. Scand. {\bf 115} (2014), 229--268.


\bibitem{NR05}  M. Neal and B. Russo, {\em Representation of contractively complemented Hilbertian operator spaces on the Fock space,} Proc.\  Amer. Math. Soc. {\bf  134}, (2005), 475--485


\bibitem{P} G. K. Pedersen, {\em C*-algebras and their automorphism
groups,} Academic Press, London (1979).

\bibitem{Pisbk}   G.\ Pisier, {\em  Introduction to operator space
theory,} London Math.\ Soc.\ Lecture Note Series, 294, Cambridge
University Press, Cambridge, 2003.


\bibitem{ZWdraft}  Z. Wang, Ph.\ D.\ thesis in preparation, University of Houston (2018).
\end{thebibliography}
\end{document}